\documentclass[11pt,a4paper]{article}
\usepackage[utf8]{inputenc}
\usepackage[english]{babel}
\usepackage[T1]{fontenc}
\usepackage{lmodern}
\usepackage{amsmath}
\usepackage{accents}
\usepackage{amsfonts}
\usepackage{amssymb}
\usepackage{amsthm}
\usepackage{amscd}
\usepackage{url}
\usepackage{enumitem}
\usepackage{hhline}
\usepackage{dsfont}
\usepackage{array}
\usepackage{geometry}
\usepackage{tikz}
\usetikzlibrary{decorations.pathreplacing}
\usetikzlibrary{patterns}
\usepackage{xcolor}
\usetikzlibrary{arrows}
\usepackage{mdframed}
\usepackage{caption}
\usepackage[titletoc]{appendix}
\usepackage{csquotes}
\usepackage{hyperref}
\usepackage{float}
\usepackage{microtype}
\usepackage{graphicx}

\mdfsetup{nobreak=true}

\AtBeginDocument{
  \label{CorrectFirstPageLabel}
  
}

\usepackage[backend=biber]{biblatex}
\newcolumntype{L}{>{$}l<{$}}

\newcommand\N{\mathbb{N}}
\newcommand\M{\mathcal{M}}
\newcommand\HH{\mathcal{H}}
\newcommand\J{\mathcal{J}}

\newcommand\R{\mathbb{R}}

\newcommand\void{\text{\O}}

\newcommand\E{\mathbb{E}}
\newcommand\Id{\mathrm{Id}}

\newcommand\dd{\mathrm{d}}
\newcommand\CC{\mathcal{C}}

\newcommand\PP{\mathbb{P}}

\newcommand\Card{\mathrm{Card}}
\DeclareMathOperator\Kol{Kol}

\newcommand\T{\mathbb{T}}

\newcommand{\enstq}[2]{\left\{#1~\middle|~#2\right\}}

\DeclareMathOperator\Vol{Vol}
\DeclareMathOperator\dist{dist}

\DeclareMathOperator\Wass{Wass}
\DeclareMathOperator\Lip{Lip}
\newcommand\one{\mathds{1}}

\newcommand\numberthis{\addtocounter{equation}{1}\tag{\theequation}}
\newcommand\jump{\par\medskip}

\theoremstyle{plain}
\newtheorem{theo}{Theorem}[section]
\newenvironment{theorem}%
  {\begin{mdframed}[backgroundcolor=white]\begin{theo}}%
  {\end{theo}\par\vspace{0.1cm}\end{mdframed}}
  
\theoremstyle{plain}
\newtheorem{coro}[theo]{Corollary}
\newenvironment{corollary}%
  {\begin{mdframed}[backgroundcolor=white]\begin{coro}}%
  {\end{coro}\par\smallskip\end{mdframed}}
  
\theoremstyle{plain}
\newtheorem{lemm}[theo]{Lemma}
\newenvironment{lemma}%
  {\begin{mdframed}[backgroundcolor=white]\begin{lemm}}%
  {\end{lemm}\par\vspace{0.cm}\end{mdframed}}

\theoremstyle{definition}
\newtheorem{remark}[theo]{Remark}

\makeatletter
\def\blfootnote{\gdef\@thefnmark{}\@footnotetext}
\makeatother
  
\begin{document}

\title{\Huge{Almost sure asymptotics for \\ Riemannian random waves}}
\author{Louis Gass}
\maketitle
\blfootnote{Univ Rennes, CNRS, IRMAR - UMR 6625, F-35000 Rennes, France.}
\blfootnote{This work was supported by the ANR grant UNIRANDOM, ANR-17-CE40-0008.}
\blfootnote{Email: louis.gass(at)ens-rennes.fr}

\begin{abstract}We consider the Riemannian random wave model of Gaussian linear combinations of Laplace eigenfunctions on a general compact Riemannian manifold. With probability one with respect to the Gaussian coefficients, we establish that, both for large and small band models, the process properly rescaled and evaluated at an independently and uniformly chosen point $X$ on the manifold, converges in distribution under the sole randomness of $X$ towards an universal Gaussian field as the frequency tends to infinity. This result extends the celebrated central limit Theorem of Salem--Zygmund for trigonometric polynomials series to the more general framework of compact Riemannian manifolds. We then deduce from the above convergence the almost-sure asymptotics of the nodal volume associated with the random wave. To the best of our knowledge, in the real Riemannian case, these asymptotics were only known in expectation and not in the almost sure sense due to the lack of sufficiently accurate variance estimates. This in particular addresses a question of S. Zelditch regarding the almost sure equidistribution of nodal volume.
\end{abstract}

\setcounter{tocdepth}{2}
\tableofcontents\jump

\newpage

\section{Introduction and main results}

\subsection{Introduction}

The central limit theorem by Salem and Zygmund \cite{Sa54} asserts that, when properly rescaled and evaluated at a uniform random point on the circle, a generic real trigonometric polynomial converges in distribution towards a Gaussian random variable. This classical result was recently revisited in \cite{An19} where the authors established both a quantitative version and a functional version of Salem--Zygmund theorem and then use these results to deduce the almost sure asymptotics of the number of zeros of random trigonometric polynomials with symmetric coefficients. 
The goal of the present article is to extend the latter results to the general Riemannian framework and more particularly to the so-called Riemannian random wave model, where random trigonometric polynomials are naturally replaced by random, Gaussian, linear combinations of Laplace eigenfunctions. \jump

The study of nodal sets associated with Laplace eigenfunctions is the object of a vast literature, in particular thanks to Yau's conjecture, see \cite{Ya82,Do88} and \cite{Ma16,Log18,Lo18} for recent breakthroughs. {The introduction of probabilistic models in this context has numerous motivations in mathematical physics among which quantum chaos heuristics \cite{Ze10}, Berry's conjecture \cite{Be77}, which suggests that quantum chaotic eigenfunctions asymptotically behave like Euclidean random waves, and in percolation, as attests the Bogomolny-Schmit conjecture \cite{Bog02}.} The most common probabilistic model then consists in considering random linear combinations of Laplace eigenfunctions, whose coefficients are independent and identically distributed standard Gaussian variables, see for instance \cite{Ru08}, \cite{Wi09} in the case of toral and spherical harmonics or \cite{Ze09} for the case of a general Riemannian manifold. \jump

There is a vast literature on the asymptotic behavior of nodal observables in the complex setting, see for instance \cite{Ze99} and the references therein. In the complex case, the fast decay of the limit covariance kernels appearing in the models provides strong concentration estimates that naturally lead to almost sure asymptotics. There exist a significant technical difference between the complex case and the real case, explaining why the Riemannian analogue of long known results in the complex domain have evaded proof. The covariance kernels in the Riemannian setting are oscillatory and of rather slow decay, while in the complex domain they have exponential decay away from the diagonal. This is the main reason why, in the case of real Riemannian manifolds, most results in the literature concern the asymptotics of nodal observables in expectation \cite{Ca16,Le16}. Coupling Borel--Cantelli Lemma with available concentration/variance estimates, one can then provide almost sure asymptotics, but only along sufficiently decreasing subsequences (see \cite{Ca16}) or Proposition 18 in \cite{Ma18} in the case of the sphere. In this paper, we establish the almost sure asymptotics of the nodal volume associated to Riemannian random waves, without considering a subsequence of eigenvalues.\jump

Indeed, we consider here a generic Gaussian combination of Laplace eigenfunctions and this combination being fixed, we evaluate it at a uniform and independent random point on the manifold. Under the sole randomness of this evaluation point, we then prove that when properly normalized and localized in the neighborhood of the point, the random field statistically converges towards an explicit universal Euclidean random wave, see Section \ref{sec.clt} below for precise statements. This result generalizes Salem--Zygmund's central limit theorem to the Riemannian framework. Our method is inspired by \cite{An19} and makes a crucial use of Weyl type estimates and some decorrelation estimates of the limit field. 
\jump
Starting from a stochastic representation formula of the nodal volume, in the spirit of Bourgain's derandomization technique \cite{Bo14,Bu16}, we then deduce from the above convergence, the almost-sure asymptotics of the nodal volume of a Riemannian random wave towards an explicit universal limit. This last result answers a question raised by S. Zelditch in \cite{Ze09} about the almost sure convergence of random nodal measure. Moreover, it allows to recover and reinforce the asymptotics in expectation obtained so far in the literature, see e.g. \cite{Le16,Ca16}. Note that our approach is only based upon the almost sure convergence in distribution of the random field and some uniform moment bounds, and it does not require any variance nor concentration estimates.

\subsection{Geometric and probabilistic setting}

In order to state our main results, let us describe the geometric and probabilistic contexts and fix our notations.

\subsubsection{Conventions}
{
Let $\N$ be the set of non-negative integers, $\R$ be the set of real numbers and $\R_+$ be the set of non-negative real numbers. Let $f,g,h$ three functions defined on an unbounded subset of $\R_+$. We use the conventional notations $o$, $O$ and $\simeq$ the following way. We say that $f(\lambda) = g(\lambda) + O(h(\lambda))$ if there is a constant $C$ such that 
\[\forall \lambda\in \R_+,\quad|f(\lambda) - g(\lambda)|\leq Ch(\lambda).\]
We say that $f(\lambda) = g(\lambda) + o(h(\lambda))$ if there is a function $\varepsilon(\lambda)$ such that 
\[\forall \lambda\in \R_+,\quad|f(\lambda) - g(\lambda)|\leq \varepsilon(\lambda)h(\lambda)\quad\text{and}\quad \lim_{\lambda\rightarrow+\infty} \varepsilon(\lambda) = 0.\]
At last, we say that $f(\lambda) \simeq g(\lambda)$ if $\lim_{\lambda\rightarrow +\infty}|f(\lambda) - g(\lambda)|=0$.}
\subsubsection{Geometric setting}

Let $(\M,g)$ be a smooth compact manifold without boundary of {dimension $d\geq 2$}. Without loss of generality we will assume that the associated volume measure $\mu$ is normalized i.e. $\mu(\M) = 1$. It is naturally equipped with the Laplace--Beltrami operator denoted $\Delta$. The second order differential operator $\Delta$ is autoadjoint and has compact resolvent. Spectral theory asserts the existence of an orthonormal basis $(\varphi_n)_{n\in\N}$ of eigenfunctions of $\Delta$ associated to the ordered eigenvalues with multiplicities $(-\lambda_n^2)_{n\in\N}$. For all $n\in\N$, we can assume that {$\lambda_n$ is non-negative} and
\[\Delta \varphi_n = -\lambda_n^2 \varphi_n\;\;\;\text{and}\;\;\;\int_\M \varphi_n^2\dd\mu = 1.\]
Given $x,y\in\M$ and $\lambda\in \R_+$, we define 
\[K_\lambda(x,y) = \sum_{\lambda_n\leq\lambda}\varphi_n(x)\varphi_n(y)\quad\text{and}\quad K_\lambda(x) := \sum_{\lambda_n\leq\lambda}\varphi_n^2(x),\numberthis\label{eq:41}\]
the two-point spectral kernel projector on the eigenspace generated by the eigenfunctions up to order $\lambda$. Integrating the function $x\mapsto K_\lambda(x)$ on $\M$ we obtain
\[K(\lambda) := \Card\enstq{n\in\N}{\lambda_n\leq \lambda} = \int_\M K_\lambda(x)\dd\mu(x),\]
the eigenvalue counting function. A fundamental tool in spectral analysis is the local Weyl law, first proved by Hörmander in \cite{Hor68}. It describes the precise asymptotics of the two-point spectral projector. Let $\sigma_d$ be the volume of the unit ball in $\R^d$:
\[\sigma_d = \frac{\pi^{d/2}}{\Gamma\left(\frac{d}{2}+1\right)},\]
and define for $x\in\R^d$ the {function $\mathcal{B}_d:\R\rightarrow\R$} by
\[
\mathcal{B}_d( \|x\|):= \frac{1}{\sigma_d}\int_{|\xi|\leq 1}e^{i\langle x,\xi\rangle}\dd \xi .
\]
It is well-defined since the right-hand since is invariant by rotation. The local Weyl law asserts that uniformly on $x,y\in\M$, as {$\lambda$ goes to infinity},
\[K_\lambda(x,y) = \frac{\sigma_d}{(2\pi)^d}\lambda^d \mathcal{B}_d(\lambda\,\dist(x,y))+O(\lambda^{d-1}).\numberthis\label{eq:24}\]
The limit kernel $\mathcal{B}_d$ only depends on the dimension $d$. It is related to the Bessel function of the first kind $\mathcal{J}$ by the formula
\[ \mathcal{B}_d(\|x\|) = \frac{1}{\sigma_d}\left(\frac{2\pi}{\|x\|}\right)^{d/2}\J_{\frac{d}{2}}(\|x\|).\]
The result of Hörmander goes beyond since the Weyl asymptotics is also true in the $\CC^\infty$ topology. For an arbitrary number of derivatives in $x$ and $y$, one has
\[\partial_{\alpha,\beta} K_\lambda(x,y) = \frac{\sigma_d}{(2\pi)^d}\lambda^d\partial_{\alpha,\beta}\left[ \mathcal{B}_d(\lambda\,\dist(x,y))\right]+O(\lambda^{d+|\alpha|+|\beta|-1}),\numberthis\label{eq:25}\]
and the remainder is also uniform on $x$ and $y$. Taking $x=y$ in the local Weyl law, 
one gets the following classical Weyl law on the number of eigenvalues of magnitude lower than $\lambda$:
\[K_\lambda(x) = \frac{\sigma_d}{(2\pi)^d}\lambda^d + O(\lambda^{d-1})\quad\text{and}\quad K(\lambda) = \frac{\sigma_d}{(2\pi)^d}\lambda^d + O(\lambda^{d-1}),\numberthis\label{eq:35}\]
from which one can deduce a first-order asymptotics for the $n$-th eigenvalue given by
\[\lambda_n \simeq 2\pi\left(\frac{n}{\sigma_d}\right)^{1/d} .\numberthis \label{eq:3}\]

{The function $K_\lambda$ is also known as the \textit{large band kernel} since the sum in \eqref{eq:41} is over the sets of eigenvalues in $[0,\lambda]$. Now let $0<\tau\leq 1/2$ be an exponent. One can also consider a small band setting by defining 
\[k_\lambda(x,y) = \sum_{\lambda_n\in ]\lambda-\lambda^\tau,\lambda]}\varphi_n(x)\varphi_n(y)\quad\text{and}\quad k(\lambda) := \Card\enstq{n\in\N}{\lambda-\lambda^\tau\leq \lambda_n\leq\lambda}.\]
We can translate the local Weyl Law to this setting, at the cost of a worst rest. Precisely, let 
\[\mathcal{S}_d : \|x\| \mapsto \frac{1}{d\sigma_d}\int_{|\xi|= 1}e^{i\langle x,\xi\rangle}\dd \xi = \mathcal{B}_{d-2}(\|x\|).\]
One has the following asymptotics, valid in the $\CC^\infty$ topology:
\[k_\lambda(x,y) = \frac{d\sigma_d}{(2\pi)^d}\lambda^{d+\tau-1} \mathcal{S}_d(\lambda\,\dist(x,y))+O(\lambda^{d-1}).\numberthis\label{eq:26}\]
Note that the model of monochromatic random waves, corresponding to the energy windows $[\lambda,\lambda+1]$ is more challenging, since the rest in the Weyl law is of the same order as the dominant term. On manifolds with no conjugate points , one can prove a refined version of Weyl law that gives a rest of the form $O(\lambda^{d-1}/\log(\lambda))$ in \eqref{eq:35} (see \cite{Ca16}, \cite{Ke19}). It is plausible that a logarithmic decay suffices to prove Theorem \ref{thm4} and Theorem \ref{thm5}, at least in an analytic setting, but the proof would be of very different nature. The model of monochromatic random waves has been thoroughly investigated on the torus and the sphere (see for instance \cite{Wi09,Ru08, Die20}), where explicit calculations can be carried out, and refined versions of the Weyl law are available. We believe that in these particular models, some ideas of this paper can also be adapted in order to prove almost sure results concerning the convergence of nodal measure.}
%

\subsubsection{Probabilistic models}

Let us now describe our main probabilistic models, classically known as the large band and small band Riemannian random waves models. Let us consider $(a_n)_{n\geq 0}$ a sequence of independent and identically distributed standard Gaussian random variables on a probability space $(\Omega_a,\mathcal{F}_a,\PP_a)$. We will denote by $\E_a$ the associated expectation.  
The two models are defined as the following Gaussian combination of eigenfunctions:
\[f_\lambda:x\mapsto \frac{1}{\sqrt{K(\lambda)}}\sum_{\lambda_n\leq \lambda} a_n\varphi_n(x)\quad\text{and}\quad \widetilde{f}_\lambda:x\mapsto \frac{1}{\sqrt{k(\lambda)}}\sum_{\lambda_n\in]\lambda-\lambda^\tau, \lambda]} a_n\varphi_n(x).\]
We could also have introduced an intermediate band regime, see  for instance \cite{Be18}, for which $\lambda_n\in]\alpha\lambda,\lambda]$ and $\alpha\in]0,1[$, but for the simplicity of the statements, we choose to focus on the two ''extreme'' cases defined above. These processes give a probabilistic interpretation of the projector kernels introduced above since they coincide with the covariance kernel of theses processes. For all $x,y\in\M$ we have indeed
\[\E_a[f_\lambda(x)f_\lambda(y)] = \frac{K_\lambda(x,y)}{K(\lambda)}\quad\text{and}\quad\E_a[\widetilde{f}_\lambda(x)\widetilde{f}_\lambda(y)] = \frac{k_\lambda(x,y)}{k(\lambda)}.\]

Consider the canonical Euclidean space  $\R^d$ . For all $x\in\M$ we define
\[I_x : \R^d\longrightarrow T_x\M,\]
an isometry between $\R^d$ and the tangent space at $x$. We only require the mapping $x\mapsto I_x$ to be measurable. For the torus $\T^d$ we can choose for $I_x$ the canonical isometry, but in all generality there is no canonical choice (nor even a continuous choice) of a family $(I_x)_{x\in\M}$. Denoting $\exp_x$ the Riemannian exponential based at $x\in\M$ we define
\[\Phi_x := \exp_x\circ I_x.\]
This map allows us to define a rescaled and flattened version of $f_\lambda$ and $\widetilde{f}_\lambda$ (or any function on $\M$) around some point $x\in\M$ by setting
\begin{align*}
g_\lambda^x : \;\R^d &\longrightarrow \R&\widetilde{g}_\lambda^x : \;\R^d &\longrightarrow \R\\
v &\longrightarrow f_\lambda\left[\Phi_x\left(\frac{v}{\lambda}\right)\right]&v &\longrightarrow \widetilde{f}_\lambda\left[\Phi_x\left(\frac{v}{\lambda}\right)\right].
\end{align*}

In the literature the processes $g_\lambda^x$ and $\widetilde{g}_\lambda^x$ have already been studied, see for instance \cite{Be18, Ca16, Ze09}. Thanks to the Weyl law, they converge in distribution (at a fixed point $x$) towards an isotropic Gaussian process whose covariance function is given by the function $\mathcal{B}_d$ and $\mathcal{S}_d$ respectively. In particular the limit process only depends on the topological dimension $d$ and is independent of the base manifold $\M$. 
\jump

{
Let us now consider a random variable $X$ on another probability space $(\Omega_X,\mathcal{F}_X,\PP_X)$. We denote by $\mathbb P_X$ its distribution and $\E_X$ the associated expectation. We will assume $\PP_X$ to be the uniform distribution, i.e. $\PP_X=\mu$, unless otherwise specified. The random variables $(X,(a_n)_{n\geq 0})$ lives in the product probability space $(\Omega_a\times\Omega_X, \mathcal{F}_a\wedge\mathcal{F}_X,\PP_a\otimes\PP_X)$. More informally, $X$ is a an random variable independent from the sequence $(a_n)_{n\geq 0}$.} For consistency,  Randomizing on the spatial parameter $x$ we define the following processes on $\R^d$ :
\[g_\lambda^X : v\mapsto f_\lambda\left(\Phi_X\left(\frac{v}{\lambda}\right)\right)\quad\text{and}\quad \widetilde{g}_\lambda^X : v\mapsto \widetilde{f}_\lambda\left(\Phi_X\left(\frac{v}{\lambda}\right)\right), \quad v \in \mathbb R^d.\numberthis\label{eq:21}\]

\subsection{Statement of the results and outline of the proofs}

The first main result of the article is the following functional central limit theorem which generalizes \cite[Thm.~3]{An19} to the case of a general compact Riemannian manifold of dimension $d\geq 2$.
\begin{theorem}
\label{thm4}
Let $\M$ be a smooth compact manifold of dimension $d\geq 2$. Almost surely with respect to the probability $\PP_a$, the two processes $(g_\lambda^X(v))_{v\in\R^d}$ and $(\widetilde{g}_\lambda^X(v))_{v\in\R^d}$ converge in distribution under $\PP_X$ with respect to the $\CC^\infty$ topology, towards isotropic Gaussian processes $(g_\infty(v))_{v\in\R^d}$ and $(\widetilde{g}_\infty(v))_{v\in\R^d}$ with respective covariance functions
\[\E_X[g_\infty(u)g_\infty(v)] = \mathcal{B}_d(\|u-v\|)\quad\quad\text{and}\quad\quad \E_X[\widetilde{g}_\infty(u)\widetilde{g}_\infty(v)] = \mathcal{S}_d(\|u-v\|).\]

The result still holds if $X$ is a random variable on $\M$ with a bounded density with respect to the volume measure $\mu$.
\end{theorem}

Let us emphasize that in the literature these kind of results are known only under Gaussian expectation. The concentration result obtained in \cite{Ca16} allows the authors to prove a similar result up to a subsequence of polynomial growth. Our result is new in the sense that the sole randomization on the uniform random variable $X$ suffices to recover the asymptotic behavior of $f_\lambda$ (without extracting a subsequence), and open the door to almost-sure results concerning functionals of $f_\lambda$, as demonstrates the next Theorem \ref{thm5} concerning almost-sure asymptotics of the nodal volume.\jump

The proof of Theorem \ref{thm4} is the object of the next Section \ref{sec.clt} and it is based upon convergence of characteristic functions. Taking the expectation under $\PP_a$, the Gaussian framework allows us to make -- technical but explicit -- computations of characteristic functions. By a Borel--Cantelli argument we recover an almost sure convergence under $\PP_a$. The proof could certainly be applied to more general settings as it uses mostly the following two main ingredients :
\begin{itemize}
\item The local Weyl law, which gives the limit distribution of $g_\lambda^x$ (as a Gaussian process) towards the Gaussian process $g_\infty$.
\item The statistical decorrelation of Lemma \ref{lemme1bis}, which roughly states that if $X$ and $Y$ are independent uniform random variables on $\M$, then the associated Gaussian processes $g_\lambda^X$ and $g_\lambda^Y$ statistically decorrelate as $\lambda$ goes to $+\infty$. It is a consequence of the decaying rate of the limit kernel $\mathcal{B}_d$.
\end{itemize}

As usual, a proof of convergence for stochastic processes splits into two parts. The convergence of finite dimensional distributions given by Theorem \ref{thm0}, and a tightness property given by Theorem \ref{thm2}.\jump

The second main result of the article is the following almost-sure asymptotics of the nodal volume associated with the random fields $f_\lambda$ and $\widetilde{f}_\lambda$. Almost surely, for $\lambda$ large enough, the nodal sets $\lbrace f_\lambda=0\rbrace$ and $\lbrace \widetilde{f}_\lambda=0\rbrace$ are random smooth submanifolds of codimension one, {thanks to Bulinskaya Lemma (see \cite[p.~34]{Az09}). We denote by $\HH^{d-1}$ the $(d-1)-$dimensional Hausdorff measure, and we set
\[\delta := (\sigma_d)^{-1/d}\quad\quad\text{and}\quad\quad B := B(0,\delta), \numberthis \label{eq:4}\]
the Euclidean ball centered at zero with radius $\delta$. Recall that the quantity $\sigma_d$ is the volume of the unit ball in $\R^d$. The parameter $\delta$ is chosen such that the ball $B$ has unit volume.}
\begin{theorem}
\label{thm5}
Almost surely with respect to the probability $\PP_a$,
\[\lim_{\lambda\rightarrow+\infty} \frac{\HH^{d-1}(\lbrace f_\lambda=0\rbrace)}{\lambda} = \E_X[\HH^{d-1}(\lbrace g_\infty=0\rbrace\cap {B})],\] 
and
\[\lim_{\lambda\rightarrow+\infty} \frac{\HH^{d-1}(\lbrace \widetilde{f}_\lambda=0\rbrace)}{\lambda} = \E_X[\HH^{d-1}(\lbrace \widetilde{g}_\infty=0\rbrace\cap {B})].\] 
\end{theorem}

This result improves the result \cite[Thm.~1]{Ze09} or \cite[Thm.~1.1]{Le16} about the convergence of nodal volume in expectation under $\PP_a$. Passing from an almost-sure convergence to a convergence in expectation is a short corollary of our proof (see Corollary \ref{coro1}). This result can even be strengthened to the weak convergence of nodal measure. Let $(\mu_\lambda)_{\lambda>0}$ (resp. $(\tilde{\mu}_\lambda)_{\lambda>0}$) be the (random) nodal measure associated to $(f_\lambda)_{\lambda>0}$ (resp. $(\tilde{f}_\lambda)_{\lambda>0}$). For a continuous function $h:\M\rightarrow\R$,
\[\int_\M h(x)\dd\mu_\lambda(x) :=  \frac{1}{\lambda}\int_{\lbrace f_\lambda = 0\rbrace} h(x)\dd\HH^{d-1}(x)\quad\text{and}\quad \int_\M h(x)\dd\tilde{\mu}_\lambda(x) :=  \frac{1}{\lambda}\int_{\lbrace \tilde{f}_\lambda = 0\rbrace} h(x)\dd\HH^{d-1}(x).\]

\begin{theorem}
\label{thm6}
Almost surely with respect to the probability $\PP_a$, the sequence $(\mu_\lambda)_{\lambda>0}$ converges weakly to the Riemannian volume measure $\mu$, up to an explicit factor.  That is, for every continuous function $h:\M\rightarrow\R$,
\[\lim_{\lambda\rightarrow+\infty} \int_\M h(x)\dd\mu_\lambda(x) = \left(\int_\M h(x)\dd\mu(x)\right)\E_X[\HH^{d-1}(\lbrace g_\infty=0\rbrace\cap {B})],\] 
and
\[\lim_{\lambda\rightarrow+\infty} \int_\M h(x)\dd\tilde{\mu}_\lambda(x) = \left(\int_\M h(x)\dd\mu(x)\right)\E_X[\HH^{d-1}(\lbrace \widetilde{g}_\infty=0\rbrace\cap {B})].\] 
\end{theorem}

This result positively answers the question raised by S. Zelditch in \cite[Cor. 2]{Ze09}, about the asymptotics of random nodal measure. In \cite{Le16} is considered the more general framework of random submanifolds (that is, intersection of independent Riemannian random waves). {It is shown that an analogous statement of Theorem \ref{thm6} for random submanifolds holds in expectation under $\PP_a$. The author believes that the strategy of proof in the present paper could be translated to the framework of \cite{Le16} to show the almost-sure asymptotics of the nodal volume of random submanifolds.}\jump

The right-hand side in Theorem \ref{thm5} can be explicitly computed by the Kac--Rice formula for random fields (see the Remark \ref{rem3}) and have in fact
\[\lim_{\lambda\rightarrow+\infty} \frac{\HH^{d-1}(\lbrace f_\lambda=0\rbrace)}{\lambda} = \frac{1}{\sqrt{\pi}}\frac{1}{\sqrt{d+2}}\frac{\Gamma\left(\frac{d+1}{2}\right)}{\Gamma\left(\frac{d}{2}\right)}\;\;\text{and}\;\;\lim_{\lambda\rightarrow+\infty} \frac{\HH^{d-1}(\lbrace \widetilde{f}_\lambda=0\rbrace)}{\lambda} = \frac{1}{\sqrt{\pi}}\frac{1}{\sqrt{d}}\frac{\Gamma\left(\frac{d+1}{2}\right)}{\Gamma\left(\frac{d}{2}\right)}.\]

\jump

The proof relies on the connection between the nodal volumes of the processes $f_\lambda$ and $g_\lambda^X$ given by Lemma \ref{lemme6} which states that for large $\lambda$,
\[\frac{\HH^{d-1}(\lbrace f_\lambda=0\rbrace)}{\lambda} \simeq \E_X[Z_\lambda],\quad\text{with}\quad Z_\lambda = \HH^{d-1}(\lbrace g_\lambda^X=0\rbrace\cap {B}).\]
By Theorem \ref{thm5} and the continuity of the random nodal volume for the $\CC^1$-topology, the continuous mapping theorem asserts that the nodal volume of $g_\lambda^X$ on the ball $B$, denoted $Z_\lambda$, converges in distribution towards the nodal volume of $g_\infty$ on the ball $B$.\jump

To recover convergence of expectations and thus Theorem \ref{thm5}, it is then sufficient to prove the uniform integrability of the family $(Z_\lambda)_{\lambda>0}$. Unfortunately the process $g_\lambda^X$ is not Gaussian under $\PP_X$, and sufficient conditions for the boundedness of power moments in the literature (as the ones given in \cite{Ar17}) are too restrictive for our purpose. The approach we use to prove finiteness of all positive moments in Theorem \ref{thm8} do not rely on the Kac--Rice formula, ill-devised for non Gaussian processes, but on more geometric considerations. \jump

Thanks to a variant of the Crofton formula given by Lemma \ref{lemme12}, we can relate the nodal volume of $g_\lambda^X$ to the anti-concentration of $g_\lambda^X$ around zero on deterministic points. The anti-concentration bound is given in Lemma \ref{lemme7} by the finiteness of a small negative moment of $g_\lambda^X$. The proof of the existence of a negative moment  uses the explicit rate of convergence of characteristic function given in Lemma \ref{lemme3}. It allows us to rewrite the convergence in term of the so-called smooth Wasserstein distance in Lemma \ref{lemme9} (following the approach in \cite{Ar17}), which is a stronger notion of convergence than the convergence in distribution.\jump

Throughout the different proofs, $C$ will denote a generic constant which does not depend on $\lambda$ nor the sequence $(a_n)_{n>0}$, and $C(\omega)$ will denote a constant which does not depend on $\lambda$ but may depend on the sequence $(a_n)_{n>0}$ (generally, a constant that comes from a Borel--Cantelli argument).\jump

At last, we will prove the above theorems in the long band regime, that is for the process $g_\lambda^X$, but the proofs apply almost verbatim in the small band regime. The minor differences arising between the two cases will be detailed in the proofs.

\section{Salem--Zygmund CLT for Riemannian random waves} \label{sec.clt}

In this section we give the proof of Theorem \ref{thm4}, and a few corollary results which will be of use in the study of the almost sure asymptotics of nodal volume in next Section \ref{sec.nodalvolume}. 
As usual, the proof of the functional convergence splits into the convergence of finite dimensional marginals and some tightness estimates.

\subsection[Finite dimensional convergence and decorrelation estimates]{Finite dimensional convergence and decorrelation estimates}

We first establish a quantitative version of the convergence of the finite dimensional marginals of $g_\lambda^X$ (resp. $\widetilde{g}_\lambda^X$) towards those of $g_\infty$ (resp. $\widetilde{g}_\infty$).We set
\begin{align*}
\eta(\lambda) = \left\lbrace\begin{array}{llll}
1 \quad\text{in the large band regime},\numberthis\label{eq:40}\\
\lambda^{1-\tau} \quad\text{in the small band regime}.\\
\end{array}
\right.
\end{align*}
Fix an integer $p\geq1$, $v=(v_1,\ldots,v_p)\in(\R^n)^p$ and
$t = (t_1,\ldots,t_p)\in\R^p$, and define in the large band regime
\[
N_\lambda(v,t) := \sum_{i=1}^p t_i \, g_\lambda^X(v_i)\quad\text{and}\quad N_\infty(v,t) := \sum_{i=1}^p t_i \, g_\infty(v_i),
\]
and respectively in the small band regime
\[
N_\lambda(v,t) := \sum_{i=1}^p t_i \, \widetilde{g}_\lambda^X(v_i)\quad\text{and}\quad N_\infty(v,t) := \sum_{i=1}^p t_i \, \widetilde{g}_\infty(v_i).
\]
We will omit the dependence in $(v,t)$ of $N_\lambda(v,t)$ and $N_\infty(v,t)$ when appropriate. Note that these linear combinations are Gaussian random variables under $\PP_a$. We prove that the characteristic function of $N_{\lambda}$ under $\mathbb P_X$ converges to the one of $N_{\infty}$ as $\lambda$ goes to infinity.
\begin{theorem}
\label{thm0}
Almost surely with respect to the probability $\PP_a$, in the large band regime,
\[ \forall t\in\R^p,\forall v\in \R^p,\;
\lim_{\lambda \rightarrow +\infty} \mathbb E_X\left[e^{iN_\lambda(v,t)}\right] = \E_X\left[e^{iN_\infty(v,t)}\right]  = \exp \left(-\frac{1}{2}\sum_{i,j=1}^p t_i t_j \mathcal B_d(||v_i -v_j||) \right), 
\]
and in the small band regime,
\[ \forall t\in\R^p,\forall v\in \R^p,\;
\lim_{\lambda \rightarrow +\infty} \mathbb E_X\left[e^{iN_\lambda(v,t)}\right] = \E_X\left[e^{iN_\infty(v,t)}\right]  = \exp \left(-\frac{1}{2}\sum_{i,j=1}^p t_i t_j \mathcal S_d(||v_i -v_j||) \right). 
\]
\end{theorem}

Since $N_\lambda(v,t)$ is a Gaussian random variable under $\PP_a$, the explicit formula for the characteristic function of a Gaussian variable gives
\[\E_a\left[e^{iN_\lambda(v,t)}\right] = e^{-\frac{1}{2}\E[N_\lambda(v,t)^2]},
\quad\text{and similarly,}\quad
\E_X\left[e^{iN_\infty(v,t)}\right] = e^{-\frac{1}{2}\E[N_\infty(v,t)^2]}.\]
In order to quantify the convergence rate, for any integer $q>0$, we set
\begin{equation}
\begin{array}{ll}
\Delta_\lambda^{(q)}  & := \E_a\left[\left|\E_X\left[e^{iN_\lambda(v,t)}\right] - \E_X\left[e^{iN_\infty(v,t)}\right]\right|^{2q}\right].
\end{array}\label{eq:1}
\end{equation}
Let $K$ be a compact subset of $\R^d$. In the following we will assume that the vectors $v_1,\ldots,v_p$ belong to $K$. Recall the definition of $\eta(\lambda)$ in \eqref{eq:40}.
\begin{theorem}
\label{thm1}
There is a constant $C$ depending only on $\M$, $K$ and $q$, such that
\begin{equation}\label{eq.decor}
\Delta_\lambda^{(q)} \leq C(1+\|t\|)^{4q}\left(\frac{\eta(\lambda)}{\lambda}\right)^q.
\end{equation}
\end{theorem}

The proof of Theorem \ref{thm0} is a direct consequence of the second assertion in Theorem \ref{thm10} at the end of this section, but observe that Theorem \ref{thm1} implies a weak version of Theorem \ref{thm0} and gives the core idea of the proof. Indeed, let us recall from Equation \eqref{eq:3} that the sequence $(\lambda_n)_{n\geq0}$ of eigenvalues grows as $Cn^{1/d}$. Fix some $t\in\R^p$ and let $\varepsilon>0$ be a small parameter. Markov inequality implies that
\begin{align*}
\PP_a\left(\left|\E_X\left[e^{iN_{\lambda_n}(v,t)}\right] - \E_X\left[e^{iN_\infty(v,t)}\right]\right|>\lambda_n^{\varepsilon}\sqrt{\frac{\eta(\lambda_n)}{\lambda_n}}\right)&\leq \frac{\Delta_{\lambda_n}^{(q)}}{\lambda_n^{2q\varepsilon}}\left(\frac{\lambda_n}{\eta(\lambda_n)}\right)^q = O\left(n^{-2q\varepsilon/d}\right).
\end{align*}
For $q>d/(2\varepsilon)$, the left-hand term is summable and Borel--Cantelli Lemma implies the existence a constant $C(\omega,v,t)$ such that
\[
\left|\E_X\left[e^{iN_\lambda(v,t)}\right] - \E_X\left[e^{iN_\infty(v,t)}\right]\right|\leq C(\omega,v,t)\frac{\sqrt{\eta(\lambda)}}{\lambda^{\frac{1}{2}-\varepsilon}}.\numberthis\label{eq:38}
\]
In particular, for a fixed $t\in\R^p$ and $v\in \R^p$, this proves the convergence in distribution of $N_\lambda(v,t)$ towards $N_\infty(v,t)$, almost surely with respect to the probability $\PP_a$. Note that Theorem \ref{thm0} states that the convergence holds almost surely under $\PP_a$, simultaneously for all $t\in\R^d$ and $v\in K^p$, and thus requires the inversion of quantifiers. We deal with this issue in Theorem \ref{thm10} at the end of Section \ref{sec.clt}, which makes explicit the dependence of $C(\omega,v,t)$ in Equation \eqref{eq:38} with respect to $v$ and $t$.
\jump
\begin{proof}[Proof of Theorem \ref{thm1}]
Define
\begin{equation}
\widetilde{\Delta}_\lambda^{(q)} := \E_a\left[\left|\E_X\left[e^{iN_\lambda(v,t)}\right] - \E_a\E_X\left[e^{iN_\lambda(v,t)}\right]\right|^{2q}\right] .\label{eq:0}
\end{equation}
By triangular inequality, we have
\[
\Delta_\lambda^{(q)} \leq 4^q\left(\widetilde{\Delta}_\lambda^{(q)} + \left|e^{-\frac{1}{2}\E_X[N_\infty(v,t)^2]}-\E_X\left[e^{-\frac{1}{2}\E_a[N_\lambda(v,t)^2]}\right]\right|^{2q}\right).
\]
Using the $1-$Lipschitz regularity of $x \mapsto e^{-x}$, we then get
\[
\Delta_\lambda^{(q)} \leq 4^q \widetilde{\Delta}_\lambda^{(q)}  + 4^{q-1}\left|\E_X\left[N_\infty(v,t)^2\right]-\E_X\E_a\left[N_\lambda(v,t)^2\right]\right|^{2q} \numberthis \label{eq.delta0}.
\]
The last term in Equation \eqref{eq.delta0} can be evaluated as follows. The following direct computation is done is the large band regime with limit kernel $\mathcal{B}_d$, but it remain true in the small band regime with limit kernel $\mathcal{S}_d$. We have first
\begin{align*}
\E_a[N_\lambda^2] &= \E_a\left[\left(\sum_{i=1}^pt_i g_\lambda^X(v_i)\right)^2\right] \\
&= \sum_{i,j=1}^pt_it_j \frac{1}{K(\lambda)}\sum_{\lambda_n\leq \lambda} \varphi_n\left[\Phi_X\left(\frac{v_i}{\lambda}\right)\right]\varphi_n\left[\Phi_X\left(\frac{v_j}{\lambda}\right)\right]\\
&= \sum_{i,j=1}^pt_it_j \frac{K_\lambda\left(\Phi_X\left(\frac{v_i}{\lambda}\right),\Phi_X\left(\frac{v_j}{\lambda}\right)\right)}{K(\lambda)}.
\end{align*}
Using Weyl law and the fact that $v$ lives in a compact set, we obtain
\[
\begin{array}{l}
\displaystyle{\left|\E_a\left[N_\lambda^2\right]-\E_X[N_\infty^2]\right| \leq \sum_{i,j=1}^p |t_i||t_j|\left|\frac{K_\lambda\left(\Phi_X\left(\frac{v_i}{\lambda}\right),\Phi_X\left(\frac{v_j}{\lambda}\right)\right)}{K(\lambda)}- \mathcal{B}_d(\|v_i-v_j\|)\right|}\\
\displaystyle{\leq \sum_{i,j=1}^p|t_i||t_j|\left| \mathcal{B}_d\left[\lambda\dist\left( \Phi_X\left(\frac{v_i}{\lambda}\right),\Phi_X\left(\frac{v_j}{\lambda}\right)\right)\right]-\mathcal{B}_d(\|v_i-v_j\|)\right| + \|t\|^2O\left(\frac{\eta(\lambda)}{\lambda}\right)}\\
\displaystyle{\leq \sum_{i,j=1}^p|t_i||t_j| \left|\lambda\dist\left(\Phi_X\left(\frac{v_i}{\lambda}\right),\Phi_X\left(\frac{v_j}{\lambda}\right)\right)-\|v_i-v_j\|\right| + \|t\|^2O\left(\frac{\eta(\lambda)}{\lambda}\right)}.
\end{array}
\]
The last line is justified  by the fact that the function $\mathcal{B}_d$ is Lipschitz continuous. The differential of the exponential map at $0$ is the identity, which implies the following asymptotic, uniformly on $v$ in a compact subset:
\[\left|\lambda\dist\left(\Phi_X\left(\frac{v_i}{\lambda}\right),\Phi_X\left(\frac{v_j}{\lambda}\right)\right)-\|v_i-v_j\|\right|=  O\left(\frac{1}{\lambda}\right),\]
and we deduce
\[\left|\E_a\left[N_\lambda^2\right]-\E_X[N_\infty^2]\right| = \|t\|^2O\left(\frac{\eta(\lambda)}{\lambda}\right).\]
Injecting this estimate in Equation \eqref{eq.delta0}, we get
\[
\Delta_\lambda^{(q)} \leq 4^q \widetilde{\Delta}_\lambda^{(q)}  + \|t\|^{4q}O\left(\left(\frac{\eta(\lambda)}{\lambda}\right)^{2q}\right).
\]
The conclusion of Theorem \ref{thm1} then follows from the following lemma.
\end{proof}

\begin{lemma}
\label{lemme1}
There is a constant $C$ depending only on $\M$, $K$ and $q$, such that
\[
\widetilde{\Delta}_\lambda^{(q)} \leq C(1+\|t\|^{4q})\left(\frac{\eta(\lambda)}{\lambda}\right)^q.\numberthis \label{eq:6}\]
\end{lemma}

The proof of Lemma \ref{lemme1} is rather technical and for the sake of readability, it is postponed until Section \ref{append.decor1} of the Appendix. To give the reader a taste of the arguments involved, the proof is essentially based on explicit computations of characteristic functions and the key argument is the following decorrelation Lemma \ref{lemme1bis}.
With the same notations as above, let $Y$ be a uniform random variable in $\mathcal M$, independent of $X$ and of the Gaussian coefficients $(a_k)$. Let us set 
\[
N_\lambda^{X} := \sum_{j=1}^pt_j g_\lambda^{X}(v_j), \quad N_\lambda^{Y} := \sum_{j=1}^pt_j g_\lambda^{Y}(v_j),
\]
in the large band regime and respectively in the small band regime 
\[
N_\lambda^{X} := \sum_{j=1}^pt_j \widetilde{g}_\lambda^{X}(v_j), \quad N_\lambda^{Y} := \sum_{j=1}^pt_j \widetilde{g}_\lambda^{Y}(v_j).
\]
\begin{lemma}
\label{lemme1bis}
There is a constant $C$ depending only on $\M$ and $K$, such that
\[\quad\E_{X}\left[\left|\E_a\left[N_\lambda^{X}N_\lambda^{Y}\right]\right|\right] \leq C\|t\|^2\frac{\eta(\lambda)}{\lambda}.\]
\end{lemma}

\begin{proof}[Proof of Lemma \ref{lemme1bis}]
An explicit computation and the Weyl law give
\begin{align*}
\left|\E_a\left[N_\lambda^{X}N_\lambda^{Y}\right]\right| &= \left|\sum_{i,j=1}^pt_it_j\frac{1}{K_\lambda}\sum_{\lambda_n\leq \lambda} \varphi_n\left[\Phi_{X}\left(\frac{v_i}{\lambda}\right)\right]\varphi_n\left[\Phi_{Y}\left(\frac{v_j}{\lambda}\right)\right]\right|\\
&\leq \sum_{i,j=1}^p|t_i||t_j|\left|
\mathcal{B}_d\left(\lambda\,\dist(\Phi_{X}\left(\frac{v_i}{\lambda}\right),\Phi_{Y}\left(\frac{v_j}{\lambda}\right)\right)\right| + \|t\|^2O\left(\frac{\eta(\lambda)}{\lambda}\right),
\end{align*}
and the remainder is uniform on $X,Y$. Again, the above computation is done in the large band regime with limit kernel $\mathcal{B}_d$, but it holds in the small band regime with limit kernel $\mathcal{S}_d$. Define
\[c_\lambda := \lambda\,\dist\left(\Phi_{X}\left(\frac{v_i}{\lambda}\right),\Phi_{Y}\left(\frac{v_j}{\lambda}\right)\right) - \lambda\dist(X,Y).\]
By triangle inequality, $|c_\lambda|$ is bounded by $2|K|$, where $|K|$ is the diameter of the compact subset $K$ in which lives $v_1,\ldots,v_p$. It follows that
\[\left|\E_a\left[N_\lambda^{X}N_\lambda^{Y}\right]\right| \leq \sum_{i,j=1}^p|t_i||t_j| \left|\mathcal{B}_d\left(\lambda\,\dist(X,Y) + c_\lambda\right)\right| + \|t\|^2O\left(\frac{\eta(\lambda)}{\lambda}\right).\]
Taking the expectation with respect to $X$ we obtain
\begin{align*}
\E_{X}\left[\left|\E_a\left[N_\lambda^{X}N_\lambda^{Y}\right]\right|\right]&\leq
\int_M\sum_{i,j=1}^p|t_i||t_j|
\left|\mathcal{B}_d\left(\lambda\,\dist(x,Y) +  c_\lambda\right)\right|\dd \mu(x) + \|t\|^2O\left(\frac{\eta(\lambda)}{\lambda}\right)\\
&\leq \sum_{i,j=1}^p|t_i||t_j|\left(\int_{\dist(x,Y)\leq\varepsilon}
\left|\mathcal{B}_d\left(\lambda\,\dist(x,Y) +  c_\lambda\right)\right|\dd \mu(x)\right.\quad +\\
&\quad\quad\quad\quad \left.\int_{\dist(x,Y)>\varepsilon}
\left|\mathcal{B}_d\left(\lambda\,\dist(x,Y) +  c_\lambda\right)\right|\dd \mu(x)\right) + \|t\|^2O\left(\frac{\eta(\lambda)}{\lambda}\right)\\
&\leq \sum_{i,j=1}^p|t_i||t_j|(I_1 + I_2)+ \|t\|^2O\left(\frac{\eta(\lambda)}{\lambda}\right),\numberthis\label{eq:36bis}
\end{align*}
where $I_1$ and $I_2$ are the two integrals appearing in the last expression. For $\varepsilon$ small enough we can pass in local polar coordinates into the first integral $I_1$. We obtain
\begin{align*}
I_1 &\leq d\,\sigma_d\int_0^\varepsilon
\sup_{c\in[-2|K|,2|K|]}\left|\mathcal{B}_d\left(\lambda r +  c\right)\right|(1+O(r^2))r^{d-1}\dd r\\
&\leq \frac{C}{\lambda^d} \int_0^{\lambda\varepsilon} \sup_{c\in[-2|K|,2|K|]}\left|\mathcal{B}_d\left(u +  c\right)\right|u^{d-1}\dd u .\numberthis\label{eq:27bis}
\end{align*}
We use the following asymptotics for $\mathcal{B}_d$ and $\mathcal{S}_d$ at infinity:
\[\mathcal{B}_d(u) = Cu^{-\frac{d+1}{2}}\sin\left(u-\frac{d-1}{4}\pi\right) + O\left(u^{-\frac{d+3}{2}}\right),\]
\[\mathcal{S}_d(u) = Cu^{-\frac{d-1}{2}}\sin\left(u-\frac{d-3}{4}\pi\right) + O\left(u^{-\frac{d+1}{2}}\right).\]
Injecting these asymptotics into expression \eqref{eq:27bis} we obtain
\[I_1=O\left(\frac{\eta(\lambda)}{\lambda}\right).\]
For the second integral and $\lambda$ large enough, we use the fact that $|c_\lambda|\leq 2|K|$ and the asymptotic formula for $\mathcal{B}_d$ (resp. $\mathcal{S}_d$) to obtain
\[I_2 \leq \sup_{t\geq \varepsilon}\sup_{c\in [-2|K|,2|K|]}|\mathcal{B}_d(\lambda t+c)|,\]
from which we deduce 
\[I_2= O\left(\frac{\eta(\lambda)}{\lambda}\right).\]
Finally we recover from inequality \eqref{eq:36bis} and the definition \eqref{eq:40} that
\[
\E_{X}\left[\left|\E_a\left[N_\lambda^{X}N_\lambda^{Y}\right]\right|\right] \leq C\|t\|^2\frac{\eta(\lambda)}{\lambda}.
\]
\end{proof}
\begin{remark}
\label{rem6}
if $X$ is not a uniform random variable on $M$, but has a bounded density $h$ with respect to the volume measure $\mu$, then Equation \eqref{eq:36bis} becomes
\[\E_{X}\left[\left|\E_a\left[N_\lambda^{X}N_\lambda^{Y}\right]\right|\right]\leq \|h\|_\infty \sum_{i,j=1}^p|t_i||t_j|(I_1 + I_2)+ \|t\|^2O\left(\frac{\eta(\lambda)}{\lambda}\right),\]
and the end of the proof of Lemma \ref{lemme1bis} remains unchanged. Throughout this section, this is the only difference that arises when $X$ is not a uniform distribution on $\M$.
\end{remark}

In the following application of Theorem \ref{thm4} to nodal volume, we will need finer estimates on the constant $C(\omega,v,t)$ in  Equation \eqref{eq:38}. The Borel--Cantelli Lemma does not allow to track the dependence of $C(\omega,v,t)$ with respect to the parameters $v$ and $t$. It is the content of the following theorem, proved in Appendix \ref{append.sobolev}. The proof relies of Sobolev injections in order to control the supremum norm by some $W^{k,1}$ norm, which is more convenient to work with when taking the expectation under $\PP_a$.

\begin{theorem}
\label{thm10}
Fix $\varepsilon>0$. There is a constant $C(\omega)$ depending only $K$ and $\varepsilon$, such that
\[\sup_{v\in K}\left|\E_X\left[e^{itg_\lambda^X(v)}\right] - e^{-\frac{t^2}{2}}\right|\leq C(\omega)(1+|t|^{2+\varepsilon})\left(\frac{\eta(\lambda)}{\lambda}\right)^{1/2-\varepsilon}.\]
And more generally,
\[\sup_{v\in K}\left|\E_X\left[e^{iN_\lambda(v,t)}\right] - e^{-\frac{1}{2}\E_X[N_\infty(v,t)^2]}\right|\leq C(\omega)(1+\|t\|^{2+\varepsilon})\left(\frac{\eta(\lambda)}{\lambda}\right)^{1/2-\varepsilon}.\]
\end{theorem}
 
\subsection{Tightness estimates}

We now turn to the proof of the tightness for the family $(g_\lambda^X)_{\lambda>0}$. Recall the definition of the ball $B$ in \eqref{eq:4}.
\medskip
\begin{theorem}
\label{thm2}
Almost surely with respect to the probability $\PP_a$, the family of stochastic processes $(g_\lambda^X)_{\lambda>0}$ is tight with respect to the Frechet topology on $\CC^\infty(B)$. 
\end{theorem}

The tightness in $\CC^1$ topology is sufficient for the rest of the article but the proof of $\CC^\infty$ tightness does not cost any more calculations. The proof is short once we proved the following lemma.
\begin{lemma}
\label{lemme3}
Let $p$ be a positive integer, and $\alpha$ a $d$-dimensional multi-index. There is a constant $C(\omega)$ depending only $p$ and $\alpha$ such that
\[\E_X\left[\int_{B}|\partial_\alpha g_\lambda^X(v)|^{2p}\dd v\right]\leq C(\omega).\]
\end{lemma}

The proof of Lemma \ref{lemme3} is given in the Appendix \ref{append.tight} and relies on hypercontractivity and a Borel--Cantelli argument. 
\begin{proof}[Proof of Theorem \ref{thm2}]
By Kolmogorov tightness criterion for stochastic processes (see \cite[p.~39]{Ku97}) in dimension $d$ with $\CC^\infty$ topology, it suffices to show that for every multi-index of differentiation $\beta$, for some $p>d/2$, and for all $u,v\in B$,
\[\E_X\left[\left|\partial_\beta g_\lambda^X(v)-\partial_\beta g_\lambda^X(u)\right|^{2p}\right]\leq C(\omega)\|v-u\|^{2p}.\]
We use the mean-value Theorem and Sobolev injection to get
\begin{align*}
\mathbb E_X\left[\left(\frac{\partial_\beta g_\lambda^X(v)-\partial_\beta g_\lambda^X(u)}{\|v-u\|}\right)^{2p}\right] &\leq  C\sum_{k=1}^d \mathbb  E_X\left[\left(\sup_{u\in B}\left|\partial_k \partial_\beta g_\lambda^X\right|\right)^{2p}\right]\\
&\leq  C\sum_{k=1}^d  \mathbb  E_X\left[\left(\|\partial_k\partial_\beta g_\lambda^X\|_{W^{d+1,1}}\right)^{2p}\right]\\
&\leq C \sum_{|\alpha|\leq {|\beta|+d+2}} \mathbb  E_X\left[\left(\int_{B}\left|\partial_\alpha g_\lambda^X(u)\right|\dd u\right)^{2p}\right]\\
&\leq C \sum_{|\alpha|\leq {|\beta| + d+2}} \mathbb  E_X\left[\int_{B}\left|\partial_\alpha g_\lambda^X(u)\right|^{2p}\dd u\right].
\end{align*}
From Lemma \ref{lemme3}, we have then
\[
\mathbb  E_X\left[\int_{B}\left|\partial_\alpha g_\lambda^X(u)\right|^{2p}\dd u\right] \leq C(\omega),
\]
hence the result.
\end{proof}

\section{Almost sure asymptotics of nodal volume}\label{sec.nodalvolume}

As already mentioned above, almost surely in the random coefficients, the nodal sets $\lbrace f_\lambda=0\rbrace$ and $\lbrace \widetilde{f}_\lambda=0\rbrace$ associated to the random wave models are random smooth submanifolds of codimension one. The object of this section is to give the proof of Theorem \ref{thm5} on the almost sure asymptotics of the associated nodal volume.

\subsection{A Stochastic representation formula}

The first step in the proof of Theorem \ref{thm5} consists in connecting the zeros of $f_\lambda$ (resp. $\widetilde{f}_\lambda$) to the zeros of $g_\lambda^X$ (resp. $\widetilde{g}_\lambda^X$). This is the object of the following lemma. In the following,  $\delta = (\sigma_d)^{-1/d}$ (recall its definition \eqref{eq:4}).

\begin{lemma}
\label{lemme6}
Let $f:\M\rightarrow\R$ a smooth function such that $0$ is a regular value of $f$. Then
\[\E_X\left[ \HH^{d-1}\left(\lbrace f = 0\rbrace\cap B\left(X,\frac{\delta}{\lambda}\right)\right)\right] = \left(1+O\left(\frac{1}{\lambda^2}\right)\right)\frac{\HH^{d-1}(\lbrace f = 0\rbrace)}{\lambda^d}.\]
The remainder term is uniform on the function $f$. More generally, if $h$ is a continuous function on $\M$, then 
\begin{align*}
\E_X\left[h(X)\HH^{d-1}\left(\lbrace f = 0\rbrace\cap B\left(X,\frac{\delta}{\lambda}\right)\right)\right] = \left(1+O\left(\frac{1}{\lambda^2}\right)\right)\frac{1}{\lambda^d}\int_{\lbrace f = 0\rbrace}h(x)\HH^{d-1}(x)\\
+\quad O\left(\omega_h\left(\frac{\delta}{\lambda}\right)\right)\frac{\HH^{d-1}(\lbrace f = 0\rbrace)}{\lambda^d},
\end{align*}
where $\omega_h$ denotes the modulus of continuity of $h$.
\end{lemma}
\begin{proof}
Since $0$ is a regular value of $f$, then $\HH^{d-1}(\lbrace f = 0\rbrace)<+\infty$. We have
\begin{align*}
\E_X\left[h(X)\HH^{d-1}\left(\lbrace f = 0\rbrace\cap B\left(X,\frac{\delta}{\lambda}\right)\right)\right] &= \int_\M h(x)\HH^{d-1}\left(\lbrace f = 0\rbrace\cap B\left(x,\frac{\delta}{\lambda}\right)\right)\dd\mu(x)\\
&=\int_\M\int_{\lbrace f=0\rbrace}h(x)\one_{B\left(x,\frac{\delta}{\lambda}\right)}(y)\dd\HH^{d-1}(y)\dd\mu(x)\\
&= \int_{\lbrace f=0\rbrace}\left(\int_{B\left(y,\frac{\delta}{\lambda}\right)}h(x)\dd \mu(x)\right)\dd\HH^{d-1}(y)\\
&= \int_{\lbrace f=0\rbrace}\left(h(y)+O\left(\omega_h\left(\frac{\delta}{\lambda}\right)\right)\right)\Vol_{\M}\left(B\left(y,\frac{\delta}{\lambda}\right)\right)\dd\HH^{d-1}(y)
\end{align*}
Standard comparison theorem for geodesic ball asserts that uniformly on $x$,
\begin{align*}
\Vol_{\M}\left(B\left(x,\frac{\delta}{\lambda}\right)\right) &= \Vol_{\R^d}\left(B\left(0,\frac{\delta}{\lambda}\right)\right)\left(1+O\left(\frac{1}{\lambda^2}\right)\right)\\
&= \frac{1}{\lambda^d}\left(1+O\left(\frac{1}{\lambda^2}\right)\right),
\end{align*}
from which we conclude.
\end{proof}

Note that, alternatively, we could have proved the asymptotic representation formula given by Lemma \ref{lemme6} using the closed Kac--Rice formula for manifolds in \cite{Ju19}.

\subsection{Application of the Central Limit  Theorem}

The next step in the proof of Theorem \ref{thm5} then consists in using the central limit theorem as established in Section \ref{sec.clt}. We define the mapping
\begin{align*}
\Phi_x^{(\lambda)}:B &\longrightarrow \M\\
v&\longrightarrow \Phi_x\left(\frac{v}{\lambda}\right).
\end{align*}
Choosing $f=f_\lambda$ in Lemma \ref{lemme6} and recalling the relation \eqref{eq:21} between $g_\lambda$ and $f_\lambda$, we obtain
\begin{align*}
\frac{\HH^{d-1}(\lbrace f_\lambda = 0\rbrace)}{\lambda^d} &= \E_X\left[\HH^{d-1}\left(\lbrace f_\lambda = 0\rbrace\cap B\left(X,\frac{\delta}{\lambda}\right)\right)\right]\left(1+O\left(\frac{1}{\lambda^2}\right)\right)\\
&= \E_X\left[\HH^{d-1}\left[\Phi_X^{(\lambda)}\left(\lbrace g_\lambda^X = 0\rbrace\cap B\right)\right]\right]\left(1+O\left(\frac{1}{\lambda^2}\right)\right).\numberthis\label{eq:5}
\end{align*}
The mapping $\Phi_x^{(\lambda)}$ is a diffeomorphism onto its image for $\lambda$ small enough and uniformly on $x\in\M$. The exponential map is a local diffeomorphism and its differential at zero is the identity. We deduce that the mapping $\Phi_x^{(\lambda)}$ is bi-Lipschitz, and uniformly on $x\in \M$,
\[\Lip\left(\Phi_x^{(\lambda)}\right) = \frac{1}{\lambda}\left(1+O\left(\frac{1}{\lambda}\right)\right)\quad\text{and}\quad\Lip\left((\Phi_x^{(\lambda)})^{-1}\right) = \lambda\left(1+O\left(\frac{1}{\lambda}\right)\right).\]
Using scaling properties of Hausdorff measures under bi-Lipschitz mappings we obtain
\[\HH^{d-1}\left[\Phi_X^{(\lambda)}\left(\lbrace g_\lambda^X = 0\rbrace\cap B\right)\right] = \frac{1}{\lambda^{d-1}}\HH^{d-1}\left[\lbrace g_\lambda^X = 0\rbrace\cap B\right]\left(1+O\left(\frac{1}{\lambda}\right)\right),\]
and from expression \eqref{eq:5} in follows that
\[\frac{\HH^{d-1}(\lbrace f_\lambda = 0\rbrace)}{\lambda} = \E_X\left[\HH^{d-1}\left(\lbrace g_\lambda^X = 0\rbrace\cap B\right)\right]\left(1+O\left(\frac{1}{\lambda}\right)\right),\numberthis\label{eq:22}\]
and more generally for a continuous fonction $h$ on $\M$,
\begin{align*}
\frac{1}{\lambda}\int_{\lbrace f = 0\rbrace}h(x)\HH^{d-1}(x) = \E_X\left[h(X)\HH^{d-1}\left(\lbrace g_\lambda^X = 0\rbrace\cap B\right)\right]\left(1+O\left(\frac{1}{\lambda}\right)\right)\\
+\quad O\left(\omega_h\left(\frac{\delta}{\lambda}\right)\right)\frac{\HH^{d-1}(\lbrace f = 0\rbrace)}{\lambda}.\numberthis\label{eq:23}
\end{align*}
The function $g\mapsto \HH^{d-1}\left(\lbrace g = 0\rbrace\cap B\right)$
is continuous on the set of functions such that $0$ is a regular value on $B$, endowed with the $\CC^1$ topology (see \cite[Thm.~3]{An18}). The limit process $g_\infty$ is non-degenerate since the limit kernels $\mathcal{B}_d$ and $\mathcal{S}_d$ are positive definite covariance functions, and Bulinskaya Lemma (see \cite[p.~34]{Az09}) asserts that $\PP_a$-almost surely, the point $0$ is a regular value for the process $g_infty$. Define
\[Z_\lambda := \HH^{d-1}(\lbrace g_\lambda^X=0\rbrace\cap B)\quad\quad\text{and}\quad\quad Z_\infty := \HH^{d-1}(\lbrace g_\infty =0\rbrace\cap B).\]
The continuous mapping theorem and the convergence in distribution of Theorem \ref{thm4} imply the following convergence in distribution under $\PP_X$:
\[\PP_a-\mathrm{a.s.},\quad\quad Z_\lambda\overset{\PP_X}{\Longrightarrow} Z_\infty.\numberthis\label{eq:45}\]
Theorem \ref{thm5} about convergence of nodal volume is proved if we can pass to the convergence of expectations under $\PP_X$ in \eqref{eq:45}, according to the stochastic representation formulas \eqref{eq:22}. Passing to the expectation follows from the uniform integrability (with respect to $\PP_X$) of the family of random variables $(Z_\lambda)_{\lambda>0}$. This last point is the object of the next Sections \ref{sec.negative} and \ref{sec.moment}.\jump

Let $\widetilde{X}$ be another random variable on $\M$ with a density $h$ with respect to the volume measure $\mu$. Then
\[\E_X\left[h(X)\HH^{d-1}\left(\lbrace g_\lambda^X = 0\rbrace\cap B\right)\right] = \E_{\widetilde{X}}\left[\HH^{d-1}\left(\lbrace g_\lambda^{\widetilde{X}} = 0\rbrace\cap B\right)\right].\]
We deduce that Theorem \ref{thm6} about convergence of nodal measure is proved if we can pass to the convergence of expectations under $\PP_X$ in \eqref{eq:45}, but this time with $X$ a random variable on $\M$ with a continuous density $h$ with respect to the volume measure $\mu$. Since the functional central limit Theorem (\ref{thm4}) and all the theorems of Section \ref{sec.clt} remain valid when $X$ is a random variable with bounded density $h$ (see Remark \ref{rem6}), the proof of Theorem \ref{thm6} is the same as the proof of Theorem \ref{thm5}.

\begin{remark}
\label{rem3}
The quantity $\E_X[\HH^{d-1}(\lbrace g_\infty=0\rbrace\cap {B})]$ in Theorem \ref{thm5} has an explicit value, thanks to the Kac--Rice formula. We roughly sketch the proof here (see \cite[p.~177]{Az09}  for more details). Taking the expectation in the co-area formula gives
\[\int_\R \varphi(y) \E_X\left[\vphantom{x^{x^x}}\HH^{d-1}(\lbrace g_\infty = y\rbrace\cap B )\right]\dd y = \int_{B} \E_X\left[\vphantom{\sum}\varphi(g_\infty(x))\|\nabla_x g_\infty\|\right]]\dd\mu(x).\]
The Gaussian process $g_\infty$ is stationary, hence its law does not depend on the point $x$. The Gaussian variables $g_\infty(x),\partial_1g_\infty(x),\ldots\partial_dg_\infty(x)$ are independents. Hence,
\begin{align*}
\int_\R \varphi(y) \E_X\left[\vphantom{x^{x^x}}\HH^{d-1}(\lbrace g_\infty = y\rbrace\cap B )\right]\dd y &= \Vol(B)\E_X[\varphi(g_\infty)]\E_X[\|\nabla g_\infty\|]\\
& = \E_X[\|\nabla g_\infty\|] \frac{1}{\sqrt{2\pi}}\int_\R \varphi(y) e^{-\frac{y^2}{2}}\dd y,
\end{align*}
and we deduce that for almost all $y\in\R$,
\[\E_X\left[\vphantom{x^{x^x}}\HH^{d-1}(\lbrace g_\infty = y\rbrace\cap B )\right] = 
\frac{e^{-\frac{y^2}{2}}}{\sqrt{2\pi}}\E_X[\|\nabla g_\infty\|].\]
It is actually true for all $y\in\R$, and this is the difficult part of the proof which we do not detail. An direct computation gives
\[\E_X\left[(\partial_1g_\infty)^2\right]=\ldots=\E_X\left[(\partial_dg_\infty)^2\right] = \frac{1}{d+2},\]
and
\[\E_X[\|\nabla g_\infty\|] = \sqrt{\frac{2}{d+2}}\frac{\Gamma\left(\frac{d+1}{2}\right)}{\Gamma\left(\frac{d}{2}\right)}.\]
Taking $y=0$ we deduce
\[\E_X\left[\vphantom{x^{x^x}}\HH^{d-1}(\lbrace g_\infty = 0\rbrace\cap B )\right] = \frac{1}{\sqrt{\pi}}\frac{1}{\sqrt{d+2}}\frac{\Gamma\left(\frac{d+1}{2}\right)}{\Gamma\left(\frac{d}{2}\right)},\]
When $d=1$ we recover the classical asymptotics $\frac{1}{\pi\sqrt{3}}$ for the number of real roots of a random trigonometric polynomial. For the process $\widetilde{g}_\infty$, we have
\[\E_X\left[(\partial_1g_\infty)^2\right]=\ldots=\E_X\left[(\partial_dg_\infty)^2\right] = \frac{1}{d},\]
which gives
\[\E_X\left[\vphantom{x^{x^x}}\HH^{d-1}(\lbrace \widetilde{g}_\infty = 0\rbrace\cap B )\right]=\frac{1}{\sqrt{\pi}}\frac{1}{\sqrt{d}}\frac{\Gamma\left(\frac{d+1}{2}\right)}{\Gamma\left(\frac{d}{2}\right)}.\]
\end{remark}

\subsection{Negative moment estimates for the random field} \label{sec.negative}

The uniform integrability of the volume of the nodal set can be deduced from anti-concentration of the stochastic process $g_\lambda^X$ around zero. If the manifold were real-analytic, it would be sufficient to have the finiteness of a logarithmic moment, which is the approach taken in \cite{An19}, see Remark \ref{rem4} below. Since we consider here $\CC^\infty$ manifolds, we need a stronger control, given by the following lemma.

\begin{lemma}
\label{lemme7}
Let $\nu<\frac{1}{20\tau d}$ a small exponent. There is a constant $C(v,\omega)$ such that
\[\sup_{\lambda>0} \E_X[|g_\lambda^X(v)|^{-\nu}]<C(v,\omega).\]
Let $\alpha>0$, $\varepsilon>0$ and $(v_i)_{i\in\N}$ be any sequence in $B$. There is a constant $C(\omega)$ (also depending on $\alpha$, $\varepsilon$ and the sequence $(v_i)_{i\in\N}$) such that
\[\sup_{\lambda>0} \int_1^{+\infty}\frac{1}{t^{1+\alpha+\varepsilon}}\sum_{i=0}^{\lceil t^\alpha\rceil}\E_X[|g_\lambda^X(v_i)|^{-\nu}]\,\dd t<C(\omega).\]
\end{lemma}

The second technical assertion is a refinement of the first one and will be used in the final step of the proof of uniform integrability. It compensates the fact that the constant $C(v,\omega)$ may depend on $v$, see also Remark \ref{rem5} below.\jump

The proof of Lemma \ref{lemme7} relies on the two following lemmas, which relate the speed of convergence of characteristic functions given in Theorem \ref{thm10} to more classical distances on the space of measures. The first lemma compares the Kolmogorov distance and the so-called smooth Wasserstein distance.
\begin{lemma}
\label{lemme9}
Given two random variables $X,Y$, and $\alpha\in\N$, we set
\[\Wass_{(\alpha)}(X,Y) := \sup\enstq{\left|\E[\phi(X)]-\E[\phi(Y)]\right|}{\phi\in\CC^\alpha(\R),\;\|\phi\|_\infty\leq 1,\ldots,\|\phi^{(\alpha)}\|_\infty\leq 1},\]
and
\[\Kol(X,Y) := \sup_{t\in\R} |\PP(X\leq t) - \PP(Y\leq t)|.\]
If $Y$ has a density bounded by $M$, there is a constant $C$ depending only on $M$ and $\alpha$ such that :
\[\Kol(X,Y)\leq \min\left(1,C\,\Wass_{(\alpha)}(X,Y)^{\frac{1}{\alpha+1}}\right).\]
\end{lemma}

\begin{proof}
Fix some $t\in\R$. Let $0<\varepsilon<1$, and consider $\varphi\in \CC^\alpha(\R)$ a nonincreasing function such that
\[\varphi(x) = \left\lbrace\begin{array}{ll}
1\;\;\text{if}\;\;x\leq 0\\
0\;\;\text{if}\;\;x\geq 1
\end{array}\right..\]
Define $\varphi_\varepsilon:x\mapsto \varphi((x-t)/\varepsilon)$, which is an upper $\CC^\alpha$ approximation of $\one_{]-\infty,t]}$. Then
\[\PP(X\leq t) - \PP(Y\leq t) \leq \left(\E[\varphi_\varepsilon(X)] - \E[\varphi_\varepsilon(Y)]\right) + \left(\E[\varphi_\varepsilon(Y)]-\PP(Y\leq t)\right).\]
For the first term, observe that $\|\varphi_\varepsilon^{(k)}\|_\infty = \varepsilon^{-k}\|\varphi^{(k)}\|_\infty$, and thus there is a constant $C$ such that
\[\E[\varphi_\varepsilon(X)] - \E[\varphi_\varepsilon(Y)]\leq \frac{C}{\varepsilon^\alpha}\Wass_{(\alpha)}(X,Y).\]
For the second term,
\[\E[\varphi_\varepsilon(Y)]-\PP(Y\leq t)\leq M\varepsilon.\]
We can make the same computations with a lower $\CC^\alpha$ approximation of $\one_{]-\infty,t]}$, which gives a similar lower bound on the quantity $\PP(X\leq t) - \PP(Y\leq t)$. Optimizing in $\varepsilon$ we obtain the desired bound.
\end{proof}

The second lemma relates the smooth Wasserstein distance and the rate of convergence of characteristic functions. 

\begin{lemma}
\label{lemme8}
Let $(X_n)_{n\geq 0}$ a sequence of random variables converging in distribution towards a random variable $X$. Assume that for some exponents $m\in\N$ and $\alpha\in\R_+$ there is a constant $C$ such that
\[\left|\E\left[e^{itX_n}\right]-\E\left[e^{itX}\right]\right|\leq C\frac{1+|t|^m}{n^\alpha},\]
and for some exponent $\beta >0$ :
\[\sup_{n\in\N}\E[|X_n|^\beta]< +\infty.\]
Then there is a constant $C$ depending on $m,\alpha,\beta$ such that :
\[\Wass_{(m+1)}(X_n,X) \leq Cn^{-\frac{2\alpha\beta}{2\beta+1}}.\]
\end{lemma}
A general form of the theorem can be found in \cite{Ar17}, but we prove it in Appendix \ref{append.wass} for completeness.
\begin{remark}
\label{rem1}
Denote $\Wass_{(\alpha)}^X$ is the smooth Wasserstein distance under $\PP_X$, and let $N$ be a standard Gaussian random variable. Lemma \ref{lemme8} and the rate of convergence given by Theorem \ref{thm10} imply that for some $\varepsilon>0$ the existence of a constant $C(\omega)$ independent of $v\in B$, such that
\[\Wass_{(4)}^X(g_\lambda^{X}(v),N) \leq C(\omega)\left(\frac{\eta(\lambda)}{\lambda}\right)^{1/2-\varepsilon}.\]
The moment condition is satisfied for every $\beta>0$ and uniformly in $v\in B$, by Sobolev injection and Lemma \ref{lemme3}.
\end{remark}

We are now in position to give the proof of Lemma \ref{lemme7} on the negative moment of the random field $g_\lambda$.

\begin{proof}[Proof of Lemma \ref{lemme7} ]
We define $\phi:x\mapsto |x|^{-\nu}$. Let $\phi_M$ be a $\CC^\infty(\R)$ approximation of $\phi$, which coincide on $\R\setminus [-\frac{1}{M},\frac{1}{M}]$. We can choose the function $\phi_M$ such that for all $p\in\N$, $\|\phi_M^{(p)}\|_\infty \leq C_p M^{\nu+p}$ (see Figure \ref{fig1}).
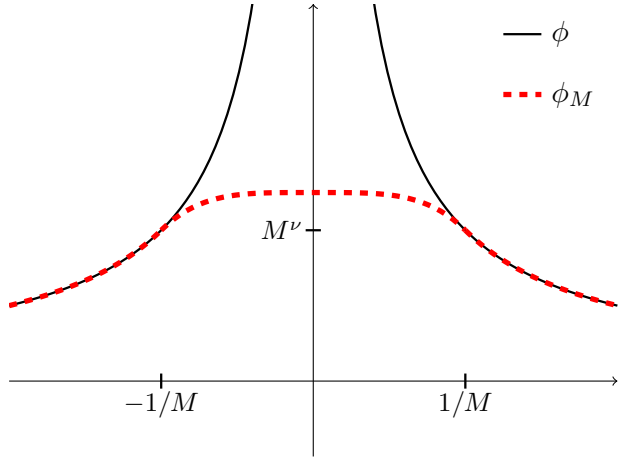
\begin{figure}[ht]
\begin{center}
\begin{tikzpicture}[xscale=2, yscale=2]
    \draw [very thin, white] (-2,-0.5) grid[step=0.5] (2,2.5);
    \clip (-2,-0.5) rectangle (2,2.5);
    \draw[->] (-2,0) -- (2,0);
    \draw [->] (0,-0.5) -- (0,2.5);
    \draw [line width=0.3mm,domain=-2:-0.1,samples=50] plot (\x,{abs(\x)^(-1)});
    \draw [line width=0.3mm,domain=0.1:2,samples=50] plot (\x,{abs(\x)^(-1)});
    \draw [dashed, line width=0.7mm,color=red,domain=-1:1,samples=50] plot (\x,{1.25-0.25*\x*\x*\x*\x});
    \draw [dashed, line width=0.7mm,color=red, domain=1:2,samples=50] plot (\x,{abs(\x)^(-1)});
    \draw [dashed, line width=0.7mm,color=red, domain=-2:-1,samples=50] plot (\x,{abs(\x)^(-1)});
    \draw (1.5,2.3) node[right]{$\phi$};
    \draw (1.5,1.9) node[right]{$\phi_M$};
    \draw [line width=0.3mm] (1.25,2.3)--(1.5,2.3);
    \draw [dashed, line width=0.7mm,color=red] (1.25,1.9)--(1.5,1.9);
    \draw (0,1) node[left]{\small{$M^\nu$}};
    \draw [line width=0.3mm] (-0.05,1)--(0.05,1);
    \draw (1,0) node[below]{\small{$1/M$}};
    \draw [line width=0.3mm] (1,-0.05)--(1,0.05);
    \draw (-1,0) node[below]{\small{$-1/M$}};
    \draw [line width=0.3mm] (-1,-0.05)--(-1,0.05);
\end{tikzpicture}
\end{center}
\caption{The functions $\phi$ and $\phi_M$.} \label{fig1}
\end{figure}

Let $N$ be a standard Gaussian random variable. We write
\[
\begin{array}{ll}
\displaystyle{\E_X\left[|g^X_\lambda(v)|^{-\nu}\right]-\E_X\left[|N|^{-\nu}\right]} & = \displaystyle{\underbrace{\E_X[\phi(g^X_\lambda(v))-\phi_M(g^X_\lambda(v))]}_{\Delta_1}+ \underbrace{\E_X[\phi_M(g^X_\lambda(v))-\phi_M(N)]}_{\Delta_2}} \\
& \displaystyle{+ \underbrace{\E_X[\phi(N)-\phi_M(N)]}_{\Delta_3}}.
\end{array}
\]
For the term $\Delta_3$, we use Cauchy--Schwarz inequality to obtain
\[\left|\E_X[\phi(N)]-\E_X[\phi_M(N)]\right|\leq \E_X\left[(\phi-\phi_M)(N)\one_{|N|\leq \frac{1}{M}}\right]\leq \sqrt{\frac{\E_X[|N|^{-2\nu}]}{M}}= \frac{C}{\sqrt{M}}.\numberthis \label{eq:12}\]
For the term $\Delta_2$, we use the smooth Wasserstein estimate in Lemma \ref{lemme8} and Remark \ref{rem1}. We have
\[|\E_X[\phi_M(g^X_\lambda(v))]-\E[\phi_M(N)]|\leq \max_{p\leq 4} \|\phi_M^{(p)}\|_\infty\Wass_{(4)}(g_\lambda^X,N) \leq CM^{\nu+4}\left(\frac{\eta(\lambda)}{\lambda}\right)^{1/2-\varepsilon}.\numberthis \label{eq:13}\]
For the more difficult term $\Delta_1$, we use Cauchy--Schwarz inequality to obtain
\begin{align*}
\left|\E_X[\phi(g^X_\lambda(v))]-\E[\phi_M(g^X_\lambda(v))]\right| &\leq \sqrt{\E_X[|g^X_\lambda(v)|^{-2\nu}]}\,.\,\sqrt{\PP_X\left(|g_\lambda^X(v)|<\frac{1}{M}\right)}.\numberthis \label{eq:14}
\end{align*}
For the right-hand term, using Kolmogorov distance and Lemma \ref{lemme9} we have
\begin{align*}
\PP_X\left(|g_\lambda^X(v)|<\frac{1}{M}\right) &\leq \PP\left(|N|<\frac{1}{M}\right) + 2\Kol(g_\lambda^X(v),N)\\
&\leq \frac{C}{M} + C\Wass_{(4)}(g_\lambda^X(v),N)^{\frac{1}{5}}\\
&\leq \frac{C}{M} + C\left(\frac{\eta(\lambda)}{\lambda}\right)^{1/10-\varepsilon}.
\end{align*}
For the left-hand term we fix $\theta = \nu d+\varepsilon$ with $\varepsilon>0$, and
\[p = \frac{1}{2\nu + \frac{\varepsilon}{d}}.\]
The exponent $p$ satisfies
\[2\nu p<1\quad\quad\text{and}\quad\quad 2\theta p>d.\]
We compute
\begin{align*}
\PP_a\left(\E_X[|g_\lambda^X(v)|^{-2\nu}|>\lambda^{2\theta}\right)&\leq \frac{\E_a\left[\E_X[|g_\lambda^X(v)|^{-2\nu}]^p\right]}{\lambda^{2p\theta}}\\
&\leq \frac{\E_X\E_a[|g_\lambda^X(v)|^{-2\nu p}]}{\lambda^{2p\theta}}.
\end{align*}
Recall that $g_\lambda$ is a Gaussian variable under $\PP_a$, whose variance approaches $1$ uniformly in $X$ and $v$. Since $2\nu p<1$ we obtain
\begin{align*}
\PP_a\left(\E_X[|g_{\lambda_n}^X(v)|^{-2\nu}]>\lambda_n^{2\theta}\right)&\leq C\frac{\E_X\left[\E_a\left[(g_{\lambda_n}^X(v))^2\right]^{-\nu p}\right]}{\lambda_n^{2\theta p}} \leq \frac{C}{\lambda_n^{2\theta p}}.
\end{align*}
Since $\lambda_n\simeq Cn^{1/d}$ the left-hand side is summable and Borel--Cantelli lemma asserts the existence of a constant $C(v,\omega)$ such that
\[\E_X[|g_{\lambda_n}^X(v)|^{-2\nu}] \leq C(v,\omega)\lambda_n^{2\theta}.\]
Finally, bounding the terms in \eqref{eq:14} we obtain
\[\left|\E_X[\phi(g^X_\lambda(v))]-\E[\phi_M(g^X_\lambda(v))]\right| \leq C(v,\omega)\lambda^{d\nu+\varepsilon}\sqrt{\frac{1}{M} + \left(\frac{\eta(\lambda)}{\lambda}\right)^{1/10-\varepsilon}}.\numberthis \label{eq:15}\]
Adding the bounds on $\Delta_1$, $\Delta_2$ and $\Delta_3$ given by the expressions \eqref{eq:12}, \eqref{eq:13} and \eqref{eq:15}, we obtain the following bound:
\[\E_X\left[|g^X_\lambda(v)|^{-\nu}\right] \leq \E_X[|N|^{-\nu}] + \frac{C}{\sqrt{M}} + CM^{\nu+4}\left(\frac{\eta(\lambda)}{\lambda}\right)^{1/2-\varepsilon}+ C(v,\omega)\lambda^{d\nu+\varepsilon}\sqrt{\frac{1}{M} + \left(\frac{\eta(\lambda)}{\lambda}\right)^{1/10-\varepsilon}}.\]
{Recall that
\begin{align*}
\eta(\lambda) = \left\lbrace\begin{array}{llll}
1 \quad\text{in the large band regime}\\
\lambda^{1-\tau} \quad\text{in the small band regime}.\\
\end{array}
\right.
\end{align*}
We then choose 
\[\nu<\frac{\tau}{20 d}\quad\text{and}\quad M = \left(\frac{\lambda}{\eta(\lambda)}\right)^{1/10}.\]}
Since a Gaussian random variable has bounded negative moments for exponents $\nu>-1$, we deduce
\[\sup_{\lambda >0}\E_X\left[|g^X_\lambda(v)|^{-\nu}\right] \leq C(v,\omega).\numberthis \label{eq:32}\]

It remains to prove the second technical part of Lemma \ref{lemme7}. We cannot directly apply the first bound since the constant obtained in \eqref{eq:32} may depend on $v$. Mimicking the previous computation, we write
\[\int_1^{+\infty}\frac{1}{t^{\alpha+2}}\sum_{i=0}^{\lfloor t^\alpha\rfloor}\E_X[|g_\lambda^X(v_i)|^{-\nu}]\,\dd t = \Delta_1 + \Delta_2 + \Delta_3.\]
Estimates \eqref{eq:12} and \eqref{eq:13} for $\Delta_1$ and $\Delta_2$ remain unchanged. For the quantity $\Delta_3$, we keep the previous notations. We have, using Markov inequality in the first line, and H\"{o}lder inequality in the second line,
\begin{align*}
\PP_a\left(\int_1^{+\infty}\frac{1}{t^{1+\alpha+\varepsilon}}\sum_{i=0}^{\lfloor t^\alpha\rfloor}\E_X[|g_\lambda^X(v_i)|^{-2\nu}]\,\dd t>\lambda^{2\theta}\right)&
\leq \frac{1}{\lambda^{2p\theta}}\E_a\left[\left(\int_1^{+\infty}\frac{1}{t^\alpha}\sum_{i=0}^{\lfloor t^\alpha\rfloor}\E_X[|g_\lambda^X(v_i)|^{-2\nu}]\,\frac{\dd t}{t^{1+\varepsilon}}\right)^p\right]\\
&\leq \frac{C}{\lambda^{2p\theta}}\int_1^{+\infty}\frac{\lfloor t^\alpha\rfloor^{p-1}}{t^{p\alpha+1+\varepsilon}}\sum_{i=0}^{\lfloor t^\alpha\rfloor}\E_X[\E_a[|g_\lambda^X(v_i)|^{-2\nu p}]\,\dd t\\
&\leq \frac{C}{\lambda^{2p\theta}}\int_1^{+\infty}\frac{\lfloor t^\alpha\rfloor^p}{t^{p\alpha+1+\varepsilon}} \dd t\\
&\leq \frac{C}{\lambda^{2p\theta}}.
\end{align*}
The end of the proof remains unchanged.
\end{proof}

\begin{remark}
\label{rem5}
The dependence in $v$ of the constant $C(v,\omega)$ given in Equation \eqref{eq:32} is not entirely satisfactory, and is a consequence of Borel--Cantelli lemma in Equation \eqref{eq:15}. We were not able to give a bound on the quantity
\[\E_a\left[\sup_{v\in B}\E_X\left[|g_\lambda^X(v)|^{-\nu}\right]\right].\]
It does not impact the rest of the article since the second part of Lemma \ref{lemme8} suffices to carry out our computations, but let us give a little more insight about what happens from a measure-theoretic point of view.\jump

The Sobolev trick we used before to obtain the uniformity on $v$ does not apply here due to the lack of regularity of the function $x\mapsto|x|^{-\nu}$. Nevertheless it may happen in particular cases that we can recover uniformity. If we are on a torus $\T^d$ endowed with any flat metric, we can choose for the isometry $I_x$ the canonical embedding into $\R^d$ and the mapping $\Phi_x$ is the usual sum. If $X$ is a uniform random variable on $\T^d$, then so is $X+v$ for any $v\in\T^d$. It follows that under $\PP_X$ and for all $v,v'\in\R^d$,
\[g_\lambda^X(v) \overset{\mathcal{L}}{=} g_\lambda^{X}(v'),\]
and quantities such as $\E_X\left[|g_\lambda^X(v)|^{-\nu p}\right]$ do not depend on $v$, which gives the uniformity in $v$. Denote by $\mu_v$ the pushforward of the measure $\mu$ under the mapping $x\mapsto \Phi_x(v)$. For all $f\in\CC^0(\M)$,
\[\int_\M f(\Phi_x(v))\dd \mu(x) = \int_\M f(x)\dd\mu_v(x).\]
In the torus case, $\mu_v$ is the canonical measure and does not depend on the parameter $v$. In all generality, few can be said about $\mu_v$. It does not always admit a density with respect to the Riemannian measure since the function $x\mapsto \Phi_x(v)$ may have support on a $1$-dimensional subspace for an ill-chosen choice of isometries $(I_x)_{x\in\M}$. Nevertheless, if the measure $\mu_v$ has a density $h_v$ belonging to $L^p(\M)$ space for some $p>1$ and uniformly on $v\in B$, then
\begin{align*}
\int_\M|g^X_\lambda(v)|^{-\nu}\dd\mu(x) &= \int_\M |f_\lambda(x)|^{-\nu} h_{\frac{v}{\lambda}}(x)\dd \mu(x)\\
&\leq \left(\int_\M |f_\lambda(x)|^{-\nu q}\right)^{1/q}\sup_{v\in B}\|h_v\|_p\quad\quad\text{with}\;\;\frac{1}{p}+\frac{1}{q}=1.
\end{align*}
Taking the expectation under $\PP_a$, and choosing $\nu<\frac{1}{q}$, we obtain
\[\E_a\left[\sup_{v\in B}\E_X\left[|g_\lambda^X(v)|^{-\nu}\right]\right]\leq \sup_{v\in B}\|h_v\|_p \int_\M\E_a\left[|f_\lambda(x)|^{-\nu q}\right]^{1/q}\dd x <+\infty.\]
At last, given a smooth compact Riemannian manifold $\M$, it is always possible to construct a family of isometries $(I_x)_{x\in\M}$ such that the family $(\mu_v)_{v\in B}$ has a density uniformly bounded in $L^\infty$.
\end{remark}

\subsection{Uniform moment estimates for the nodal volume} \label{sec.moment}

In order to complete the proof of the uniform integrability of nodal volume, we now introduce a geometric lemma which relates the nodal volume of a function to the number of zero of this function on a straight line passing through predefined points. It is a variant of the Crofton formula (see \cite{Al07} for a general presentation of the various Crofton formulæ), and a $d$-dimensional extension of \cite[Thm.~6]{An19}.

\begin{lemma}
\label{lemme12}
Let $E$ be a $\CC^2$-hypersurface in $\R^d$, intersecting the cube $D = [0,a]^d$. Assume that $E$ has bounded curvature on $D$. Then there exists a segment $S$ passing through one of the vertices of the cube $D$ and such that
\[\Card(E\cap S\cap D) \geq c\;\frac{\HH^{d-1}(E\cap D)}{a^{d-1}},\quad\text{with}\quad c=\frac{1}{2^{d+1}d}.\]
\end{lemma}
The proof of Lemma 3.8 is postponed to Appendix \ref{append.geom}, and relies on a probabilistic method to shows the existence of a segment $S$ satisfying the conclusion of Lemma \ref{lemme12}.
\begin{remark}
\label{rem2}
If $g:\R^d\rightarrow \R$ is a smooth function and $0$ is a regular value of $g$, then $g^{-1}(\lbrace0\rbrace)$ is a smooth manifold and we can apply Lemma \ref{lemme12} to deduce the existence of a segment $S$ passing through one of the vertices of the cube $D = [0,a]^d$ and such that
\[\Card(\lbrace g=0\rbrace\cap S\cap D) \geq c\;\frac{\HH^{d-1}(\lbrace g=0\rbrace\cap D)}{a^{d-1}}. \numberthis \label{eq:16}\]
Denote $g_S$ its restriction on $S$, and suppose that $g_S$ cancels at least $p$ times at points $\textbf{w}_1,\ldots,\textbf{w}_p$. By the generalized Rolle lemma, for all $v\in S$, there exists a point $c_v$ in $S$ such that
\[|g(v)| = \frac{\prod_{j=1}^p\|v-\textbf{w}_j\|}{p!}\left|g_S^{(p)}(c_v)\right|.\]
Hence if the segment $S$ passes through the vertex $v_j$ on the cube $D$ then
\begin{align*}
|g(v_j)|&\leq \frac{\HH^1(S)^p}{p!}\sum_{|\alpha| = p}\|\partial_\alpha g\|_\infty \\
&\leq Ca^p \sum_{|\alpha| = p}\|\partial_\alpha g\|_\infty.
\end{align*}
To sum up, if we have
\[\HH^{d-1}(\lbrace g=0\rbrace\cap D) \geq \frac{a^{d-1}}{c} p,\]
then for at least one of the vertices $v_j$ of the cube,
\[|g(v_j)|\leq Ca^p \sum_{|\alpha| = p}\|\partial_\alpha g\|_\infty.\]
\end{remark}

In \cite[Thm.~5.2]{Ar19} the authors proved the finiteness of moments of nodal volume under the requirement of joint bounded density of first $k$ derivatives. This hypothesis is too strong for our purpose, since our process under $\PP_X$ depends only on the randomness of $X$ and we cannot expect a joint bounded density of the first derivatives.
\begin{theorem}
\label{thm8}
$\PP_a$-almost surely, the family of random variables $(Z_\lambda)_{\lambda>0}$ is uniformly integrable.  More precisely, for all $\gamma>0$,
\[\sup_{\lambda>0} \E_X[Z_\lambda^{1+\gamma}] \leq C_\gamma(\omega).\]
\end{theorem}

Conjointly with the convergence in distribution of the nodal volume, it implies the convergence of all moments of $Z_\lambda$ to those of $Z_\infty$.

\begin{proof}[Proof of Theorem \ref{thm8}]
For all $A>0$ (to be fixed later),
\begin{align*}
\E_X[Z_\lambda^{1+\gamma}] &= (1+\gamma)\int_0^{+\infty} t^{\gamma}\PP_X(Z_\lambda>t)\dd t\\
&\leq C_A + (1+\gamma)\int_A^{+\infty} t^{\gamma}\PP_X(Z_\lambda>t)\dd t ,\numberthis \label{eq:18}
\end{align*}
hence we need to estimate the quantity $\PP_X(Z_\lambda>t)$ for all $t$ greater than some constant $A$. Up to embedding the ball $B$ in a cube we will consider that the vector $v$ lives in a hypercube (of size $1$ for simplicity).\jump

Consider a rectangular grid on $[0,1]^d$ of size $\frac{1}{\lfloor t^{\theta}\rfloor}$ with $\theta = 1-\varepsilon$. The hypercube $[0,1]^d$ is split into $\lfloor t^{\theta}\rfloor^d$ smaller cubes, which we number by $(D_i)_{1\leq i\leq \lfloor t^{\theta}\rfloor^d}$ (see Figure \ref{fig3}).

\begin{figure}[ht]
\begin{center}
\begin{tikzpicture}[xscale=0.5, yscale=0.5]
    \draw [very thin, gray] (0,0) grid[step=2] (10,10);
    \draw  (0,0) -- (10,0);
    \draw  (0,10) -- (10,10);
    \draw  (0,10) -- (0,0);
    \draw  (10,0) -- (10,10);
    \draw [decorate,decoration={brace,amplitude=6pt}]
    (6,-0.1) -- (4,-0.1) node [black,midway,yshift=-18pt]{$\frac{1}{\lfloor t^{\theta}\rfloor}$};
    \draw (0,0) node[below]{\small{$0$}};
    \draw (10,0) node[below]{\small{$1$}};
    \draw (0,10) node[left]{\small{$1$}};
    \draw (1.7,9) node[left]{$D_1$};
    \draw (3.7,9) node[left]{$D_2$};
    \draw (5.7,9) node[left]{$D_3$};
    \draw (7.7,9) node[left]{$\ldots$};
    \draw [line width=0.5mm,domain=0:360,samples=200, color = blue] plot ({5+3*cos(\x)*(1+0.4*cos(6*\x))},{5+3*sin(\x)*(1+0.2*cos(7*\x))*(1-0.2*cos(2*\x+60))});
    \draw [line width=0.5mm,domain=90:180,samples=200, color = blue] plot ({10+3*cos(\x)},{3*sin(\x)*(1+0.2*cos(7*\x))*(1-0.2*cos(2*\x+60))});
    \draw [line width=0.5mm,domain=-90:90,samples=200, color = blue] plot ({1*cos(\x)},{2+1*sin(\x)*(1+0.2*cos(7*\x))*(1-0.2*cos(2*\x+60))});
    \draw [line width=0.5mm,domain=270:360,samples=200, color = blue] plot ({2.5*cos(\x)},{10+2*sin(\x)*(1+0.2*cos(7*\x))*(1-0.2*cos(2*\x+60))});
    \draw [line width=0.5mm,domain=180:270,samples=200, color = blue] plot ({10+2.5*cos(\x)},{10+3*sin(\x)*(1+0.2*cos(7*\x))*(1-0.2*cos(2*\x+60))});
    \draw [line width=0.5mm,domain=0:360,samples=200, color = blue] plot ({5.5+0.6*cos(\x)},{5+0.8*sin(\x)*(1+0.2*cos(7*\x))*(1-0.2*cos(2*\x+60))});
\end{tikzpicture}
\end{center}
\caption{The grid defined on $[0,1]^d$.}
\label{fig3}
\end{figure}

\noindent Let $(v_j)_{1\leq j\leq \lceil t^{\theta}\rceil^d}$ be the vertices of the grid. For all $i\in\lbrace1,\ldots, \lfloor t^{\theta}\rfloor^d)$, let $(v_{ij})_{1\leq j\leq 2^d}$ be the vertices of the $i$-th cube, and
\[Z_\lambda^{(i)} := \HH^{d-1}(\lbrace g_\lambda^X = 0\rbrace)\cap D_i),\]
be the volume of zeros contained in the $i$-th cube. By the pigeonhole principle,
\[\left\lbrace Z_\lambda>t\right\rbrace \subset \bigcup_{i=1}^{\lfloor t^{\theta}\rfloor^d}\left\lbrace Z_\lambda^{(i)}>\frac{t}{\lfloor t^\theta\rfloor^{d}}\right\rbrace.\numberthis \label{eq:17}\]
Let $p$ be an integer to be fixed later. We use Lemma \ref{lemme12} and then Remark \ref{rem2}, keeping the same notations, to deduce that if $t\geq A$, with
\[A := \left(\frac{p}{c}\right)^{1/(1-\theta)}.\]
and $a = \lfloor t^\theta\rfloor^{-1}$, then
\[\left\lbrace Z_\lambda^{(i)}>\frac{t}{\lfloor t^\theta\rfloor^{d}}\right\rbrace \subset 
\left\lbrace Z_\lambda^{(i)}>\frac{p}{c\lfloor t^\theta\rfloor ^{d-1}} \right\rbrace \subset 
\bigcup_{j=1}^{2^d}\left\lbrace |g_\lambda^X(v_{ij})|\leq  \frac{C}{\lfloor t^\theta\rfloor^{p}} \sum_{|\alpha| = p}\|\partial_\alpha g_\lambda^X\|_\infty\right\rbrace.\]
Fix $k\geq 1$. Taking the expectation with respect to $\PP_X$ we obtain
\begin{align*}
\PP_X\left(Z_\lambda^{(i)}>\frac{t}{\lfloor t^\theta\rfloor^{d}}\right) &\leq \sum_{j=1}^{2^d}\PP_X\left(|g_\lambda(v_{ij})|\leq \frac{C}{\lfloor t^\theta\rfloor^{p}}\sum_{|\alpha| = p}\|\partial_\alpha g_\lambda^X\|_\infty\right)\\
&\leq \sum_{j=1}^{2^d}\PP_X\left(|g_\lambda^X(v_{ij})|\leq \frac{t^\varepsilon}{\lfloor t^\theta\rfloor^{p}}\right) + 2^d\PP_X\left(C\sum_{|\alpha| = p}\|\partial_\alpha g_\lambda^X\|_\infty>t^\varepsilon\right)\\
&\leq \left(\frac{t^\varepsilon}{\lfloor t^\theta\rfloor^p}\right)^\nu\sum_{j=1}^{2^d}\E_X[|g_\lambda^X(v_{ij})|^{-\nu}] + C\frac{\E_X\left[\sum_{|\alpha| = p}\|\partial_\alpha g_\lambda^X\|_\infty^k\right]}{t^{k\varepsilon}} \numberthis\label{eq:28}\\
&\leq \left(\frac{t^\varepsilon}{\lfloor t^\theta\rfloor^p}\right)^\nu\sum_{j=1}^{2^d}\E_X[|g_\lambda^X(v_{ij})|^{-\nu}] + \frac{C_k(\omega)}{t^{k\varepsilon}}.
\end{align*}
In the last line we used the Sobolev injection and estimate of Lemma \ref{lemme3} to bound the right hand side. Taking the expectation in expression \eqref{eq:17} and using the union bound we obtain
\[\PP_X(Z_\lambda>t)\leq  C_k(\omega)\frac{\lfloor t^\theta\rfloor^d}{t^{k\varepsilon}}+ \left(\frac{t^\varepsilon}{\lfloor t^\theta\rfloor^p}\right)^\nu 2^d\sum_{i=1}^{\lceil t^\theta\rceil^d}\E_X[|g_\lambda^X(v_{i})|^{-\nu}].\]
Recalling expression \eqref{eq:18} we obtain
\[\E_X[Z_\lambda^{1+\gamma}] \leq C+ C_k(\omega)\int_A^{+\infty}t^\gamma\frac{\lfloor t^\theta\rfloor^d}{t^{k\varepsilon}}\dd t+ C\int_A^{+\infty} t^\gamma\left(\frac{t^\varepsilon}{\lfloor t^\theta\rfloor^p}\right)^\nu\sum_{j=1}^{\lceil t^\theta\rceil^d}\E_X[|g_\lambda^X(v_i)|^{-\nu}]\dd t.\]
Choosing $k$ and $p$ such that
\[k>\frac{\gamma+\theta d+1}{\varepsilon}\quad\quad\text{and}\quad\quad \nu(p\theta - \varepsilon) - (d\theta+\gamma) > 1,\] 
we can apply the second part of Lemma \ref{lemme7} to deduce the existence of a constant $C(\omega)$ such that
\[\E_X[Z_\lambda^{1+\gamma}] \leq C(\omega).\]
\end{proof}

\begin{remark}
\label{rem4}
If we were in an analytic setting, we could use the same argument as the one in \cite[Thm.~9]{An19}, which roughly relies on the convergence of Taylor expansion of eigenfunctions. In the $\CC^\infty$ setting we can only apply the generalized Rolle lemma with a fixed $p$, and it explains why we used the partitioning of the cube $[0,1]^d$. A careful analysis of the proof shows that it requires a manifold of finite regularity $\CC^k$ for $k$ large enough.
\end{remark}

\begin{corollary}
\label{coro1}
For all $p\geq 1$ :
\[\lim_{\lambda\rightarrow+\infty} \frac{\E_a\left[\left(\HH^{d-1}(\lbrace f_\lambda=0\rbrace)\right)^p\right]}{\lambda^p} = \left(\frac{1}{\sqrt{\pi}}\frac{1}{\sqrt{d+2}}\frac{\Gamma\left(\frac{d+1}{2}\right)}{\Gamma\left(\frac{d}{2}\right)}\right)^p,\]
and
\[\lim_{\lambda\rightarrow+\infty} \frac{\E_a\left[\left(\HH^{d-1}(\lbrace \widetilde{f}_\lambda=0\rbrace)\right)^p\right]}{\lambda^p} = \left(\frac{1}{\sqrt{\pi}}\frac{1}{\sqrt{d}}\frac{\Gamma\left(\frac{d+1}{2}\right)}{\Gamma\left(\frac{d}{2}\right)}\right)^p.\]
\end{corollary}

\begin{proof}
Passing from almost-sure convergence to convergence in expectation is a consequence of the dominated convergence. It suffices to show that
\[\sup_{\lambda>0}\E_a\left[\left(\E_\textbf{X}[Z_\lambda]\right)^p\right]<+\infty,\]
which can bee seen by raising to the power $p$ and taking the expectation under $\PP_a$ in Equation \eqref{eq:28}. In more direct way, all the almost-sure estimates are deduced from Borel-Cantelli lemma and Markov inequality applied to the power function with arbitrary large exponent, and thus remain true under expectation. A similar argument holds for higher moments.
\end{proof}

\appendix

\section{Proof of decorrelation estimates} \label{append.decor}

In this first part of the Appendix, we give the proof of Lemma \ref{lemme1} and  Theorem \ref{thm10} stated in Section \ref{sec.clt}.

\subsection{Proof of Lemma \ref{lemme1}}\label{append.decor1}

Let $X_1,\ldots,X_{2q}$ be independents copies of $X$. The expectation with respect to the random variables $X_1,\ldots,X_{2q}$ will be noted $\E_{\textbf{X}}$. To enhance the dependence with respect to $X_k$, we set for all $k\in\lbrace 1,\ldots,q\rbrace$,
\[N_\lambda^{(k)} = \sum_{j=1}^pt_j g_\lambda^{X_k}(v_j),\]
and for all $k\in\lbrace q+1,\ldots,2q\rbrace$,
\[N_\lambda^{(k)} = -\sum_{j=1}^pt_j g_\lambda^{X_k}(v_j).\]
Then, for $k \neq l$, applying Lemma \ref{lemme1bis} with $X=X_k$ and $Y=X_l$, we have uniformly in $X_l$
\[
\quad\E_{X_k}\left[\left|\E_a\left[N_\lambda^{(k)}N_\lambda^{(l)}\right]\right|\right] = \|t\|^2O\left(\frac{\eta(\lambda)}{\lambda}\right). \numberthis \label{eq.lemme1}
\]
The following lemma, based on an explicit computation of Gaussian characteristic functions and integral Taylor formula, gives an explicit expression of $\widetilde{\Delta}_\lambda^{(q)}$, which is the object of Lemma \ref{lemme1}. For $s \in [0,1]^{2q}$, let
\[
\begin{array}{ll} f(s) & := \displaystyle{\sum_{k=1}^{2q}\E_a\left[\left(N_\lambda^{(k)}\right)^2\right] + \sum_{\substack{k,l=1\\k\neq l}}^{2q} s_ks_l\E_a\left[N_\lambda^{(k)}N_\lambda^{(l)}\right]}\\
\\
 & =\displaystyle{\E_a\left[\left(\sum_{k=1}^{2q} s_kN_\lambda^{(k)}\right)^2 \right] + \sum_{k=1}^{2q}\left(1-s_k^2\right)\E_a\left[\left(N_\lambda^{(k)}\right)^2\right]}.
 \end{array}
 \]
 
\begin{lemma}
\label{lemme4}
We have
\[\widetilde{\Delta}_\lambda^{(q)} = \E_{\textbf{X}}\left[\int_{[0,1]^{2q}} \partial_1\ldots\partial_{2q} \left(\exp\left(-\frac{1}{2}f\right)\right)(s)\dd s\right].
\]
\end{lemma}

\begin{proof}[Proof of Lemma \ref{lemme4}]
From mutual independence of the family $(X_1,\ldots,X_{2q})$,
\begin{align*}
\widetilde{\Delta}_\lambda^{(q)} = \E_{\textbf{X}}\E_a\left[\prod_{k=1}^{2q}\left(e^{iN_\lambda^{(k)}}-\E_a\left[e^{iN_\lambda^{(k)})}\right]\right)\right].
\end{align*}
We define
\[\widetilde{\Delta}_\lambda^{(q)}(s) := \E_{\textbf{X}}\E_a\left[\left(\prod_{k=1}^{2q}\E_a\left[e^{i\sqrt{1-s_k^2}N_\lambda^{(k)}}\right]\right)\left(\prod_{k=1}^{2q}\left(e^{is_kN_\lambda^{(k)}}-\E_a\left[e^{is_kN_\lambda^{(k)})}\right]\right)\right)\right].\numberthis \label{eq:7}\]
Developing the product, and using the characteristic function of a Gaussian random variable, a direct computation shows that 
\[\widetilde{\Delta}_\lambda^{(q)}(s) := \E_{\textbf{X}}\left[\exp\left(-\frac{1}{2}\sum_{k=1}^{2q}\E_a\left[\left(N_\lambda^{(k)}\right)^2\right]\right)\!\!\sum_{A\subset \lbrace 1,\ldots,2q\rbrace}\!\! (-1)^{|A|}\exp\left(-\frac{1}{2}\sum_{\substack{k,l\in A\\k\neq l}} s_ks_l\E_a\left[N_\lambda^{(k)}N_\lambda^{(l)}\right]\right)\right].\numberthis \label{eq:8}\]
If $s_k = 0$ for some $k\in\lbrace1,\ldots,2q\rbrace$, then from expression \eqref{eq:7},
\[
\widetilde{\Delta}_\lambda^{(q)}(s)=0.
\]
In other words, the function $s\mapsto \widetilde{\Delta}_\lambda^{(q)}(s)$ cancels if one of its coordinates is zero. By integral Taylor formula,
\[\widetilde{\Delta}_\lambda^{(q)} = \widetilde{\Delta}_\lambda^{(q)}(1,\ldots,1) = \int_{[0,1]^{2q}} \partial_1\ldots\partial_{2q} \widetilde{\Delta}_\lambda^{(q)}(s)\dd s.\]
But from expression \eqref{eq:8}, the only term that depends on all coordinates (and thus won't be canceled after differentiation) is the term corresponding to $A= \lbrace1,\ldots,2q\rbrace$, which is
\[\E_{\textbf{X}}\left[\exp\left(-\frac{1}{2}\sum_{k=1}^{2q}\E_a\left[\left(N_\lambda^{(k)}\right)^2\right]\right)\exp\left(-\frac{1}{2}\sum_{\substack{k,l=1\\k\neq l}}^{2q} s_ks_l\E_a\left[N_\lambda^{(k)}N_\lambda^{(l)}\right]\right)\right]=\E_\textbf{X}\left[\exp\left(-\frac{1}{2}f(s)\right)\right].\]
\end{proof}
Now, for a set $A$, denote $\Pi(A)$ the collection of partitions of $A$ into groups of two elements. A direct computation shows that
\begin{align*}
\partial_1\ldots\partial_{2q} \left(\exp\left(-\frac{1}{2}f\right)\right) = 
\exp\left(-\frac{1}{2}f\right)\sum_{\substack{A\subset\lbrace1,\ldots,2q\rbrace\\|A| \;\text{even}}}(-1)^\frac{|A|}{2}&\left(\sum_{B\in \Pi(A)} \prod_{(k,l)\in B}\E_a\left[N_\lambda^{(k)}N_\lambda^{(l)}\right]\right)\\
&\quad\quad\times\prod_{k\in A^c}\left(\sum_{\substack{l=1\\k\neq l}}^{2q}s_l\E_a\left[N_\lambda^{(k)}N_\lambda^{(l)}\right] \right)\;.
\end{align*}
We deduce, using the mutual independence of the family $X_1,\ldots,X_{2q}$ and the estimate \ref{eq.lemme1},
\begin{align*}
\widetilde{\Delta}_\lambda^{(q)} &\leq C\E_{\textbf{X}}\left[
\sum_{\substack{A\subset\lbrace1,\ldots,2q\rbrace\\|A| \;\text{even}}}\left(\sum_{B\in \Pi(A)} \prod_{(k,l)\in B}\left|\E_a\left[N_\lambda^{(k)}N_\lambda^{(l)}\right]\right|\right)\prod_{k\in A^c}\left(\sum_{\substack{l=1\\k\neq l}}^{2q}\left|\E_a\left[N_\lambda^{(k)}N_\lambda^{(l)}\right]\right| \right)\right]\\
&\leq C\sum_{\substack{A\subset\lbrace1,\ldots,2q\rbrace\\|A| \;\text{even}}}\left(\sum_{B\in \Pi(A)} \prod_{(k,l)\in B}\E_{\textbf{X}}\left[\left|\E_a\left[N_\lambda^{(k)}N_\lambda^{(l)}\right]\right|\right]\right)\E_{\textbf{X}}\left[\prod_{k\in A^c}\left(\sum_{\substack{l=1\\k\neq l}}^{2q}\left|\E_a\left[N_\lambda^{(k)}N_\lambda^{(l)}\right]\right| \right)\right]\\
&\leq C\sum_{\substack{A\subset\lbrace1,\ldots,2q\rbrace\\|A| \;\text{even}}}\left(\|t\|^2\frac{\eta(\lambda)}{\lambda}\right)^{\frac{|A|}{2}}\E_{\textbf{X}}\left[\prod_{k\in A^c}\left(\sum_{\substack{l=1\\k\neq l}}^{2q}\left|\E_a\left[N_\lambda^{(k)}N_\lambda^{(l)}\right]\right| \right)\right]\;.\numberthis \label{eq:2}
\end{align*}
To give a bound on the right-hand term and thus establish Lemma \ref{lemme1}, we use the following lemma whose proof again relies on the decorrelation estimates of Lemma \ref{lemme1bis}.
\begin{lemma}
\label{lemme5}
There is a constant $C$ depending only on $\M$, $K$ and $q$ such that
\[\quad\E_{\textbf{X}}\left[\prod_{k\in A^c}\left(\sum_{\substack{l=1\\l\neq k}}^{2q}\left|\E_a\left[N_\lambda^{(k)}N_\lambda^{(l)}\right]\right| \right)\right] \leq C(1+\|t\|^4)^{q-\frac{|A|}{2}}\left(\frac{\eta(\lambda)}{\lambda}\right)^{q-\frac{|A|}{2}}.\]
\end{lemma}

\begin{proof}[Proof of Lemma \ref{lemme5}]
Assume without loss of generality that $A^c = \lbrace 1,\ldots,2m\rbrace$. We compute
\begin{align*}
\E_\textbf{X}\left[\prod_{k=1}^{2m}\left(\sum_{\substack{l=1\\k\neq l}}^{2q}\left|\E_a\left[N_\lambda^{(k)}N_\lambda^{(l)}\right]\right| \right)\right] = \sum_{\substack{l_1,\ldots,l_{2m}=1\\l_k\neq k}}^{2q} \underbrace{\E_\textbf{X}\left[\prod_{k=1}^{2m}\left|\E_a\left[N_\lambda^{(k)}N_\lambda^{(l_k)}\right]\right|\right]}_{\Delta_\ell}.
\end{align*}
Now fix $\ell = (l_1,\ldots,l_{2m})$. Consider the following graph $G_\ell$ with vertices in $\lbrace 1,\ldots,2q\rbrace$: two vertices $k$ and $l$ are connected if the term $\left|\E_a\left[N_\lambda^{(k)}N_\lambda^{(l)}\right]\right|$ appears into the expression $\Delta_\ell$. If the graph $G_\ell$ is disconnected, we can use independence of the random variables $X_1,\ldots X_{2q}$, and we are left to show the aforementioned bound for connected graphs. Thanks to Weyl law, there is a constant $C$ such that
\[\left|\E_a\left[N_\lambda^{(k)}N_\lambda^{(l)}\right]\right| \leq C\|t\|^2,\]
and we can assume (up to bounding one of the terms in the product) that $G_\ell$ is a tree (with $2m-1$ edges). Suppose without loss of generality that $1$ is a leaf of the tree attached to $2$. Recalling the definition \eqref{eq:40} of $\eta(\lambda)$, one has
\begin{align*}
\Delta_\ell & \leq C\|t\|^2\, \E_{X_2,\ldots ,X_{2q}}\left[\E_{X_1}\left[\left|\E_a\left[N_\lambda^{(1)}N_\lambda^{(2)}\right]\right|\right]\prod_{\substack{(k,l)\in G_\ell \\ (k,l)\neq(1,2)}}\left|\E_a\left[N_\lambda^{(k)}N_\lambda^{(l)}\right]\right|\right]\\
&\leq C\|t\|^4O\left(\frac{\eta(\lambda)}{\lambda}\right)\, \E_{X_2,\ldots ,X_{2q}}\left[\prod_{\substack{(k,l)\in G_\ell \\ (k,l)\neq(1,2)}}\left|\E_a\left[N_\lambda^{(k)}N_\lambda^{(l)}\right]\right|\right],
\end{align*}
after the estimate \eqref{eq.lemme1}. Repeating the procedure leaf by leaf we obtain the bound
\[\E_\textbf{X}\left[\prod_{k=1}^{2m}\left|\E_a\left[N_\lambda^{(k)}N_\lambda^{(l_k)}\right]\right|\right] = \|t\|^{4m}O\left(\left(\frac{\eta(\lambda)}{\lambda}\right)^{2m-1}\right) = \|t\|^{4m}O\left(\left(\frac{\eta(\lambda)}{\lambda}\right)^m\right).\]
We used the fact that $2m-1\geq m$, with equality when $m=1$. That is, the worst case is attained for graphs $G_l$ such that their sets of edges forms partitions into pairs, for instance when $G_\ell = \lbrace(1,2),(3,4)\ldots,(2m-1,2m)\rbrace$.
\end{proof} 

\subsection{Proof of Theorem \ref{thm10}} \label{append.sobolev}

The proof of Theorem \ref{thm10} is rather technical, and relies on the following Sobolev injection for a smooth domain $\Omega$:
\[W^{d+1,1}(\Omega) \subset L_\infty(\Omega).\] 
It allows us to bound the supremum norm by the $W^{d+1,1}$ Sobolev norm, which is interchangeable with the expectation under $\PP_a$. We only detail the proof of the first assertion for simplicity. The second assertion is the generalization to the case $t\in\R^p$, and its proof follows the same lines. 
\begin{proof}[Proof of Theorem \ref{thm10}]
Let $B_K$ be a ball containing the compact $K$. Let $t\mapsto h(t)$ be a non-negative symmetric function, and non-increasing on $\R_+$. For any smooth function $f:B_K\times\R\rightarrow\R$ with $f(v,0) = 0$,
\begin{align*}
\sup_{v\in K}\sup_{t\in\R} h(t)|f(v,t)| &\leq \sup_{t\in\R} h(t) \int_{[0,t]}\sup_{v\in B_K}|\partial_t f(v,s)|\dd s\\
&\leq \sup_{t\in\R} \int_{[0,t]} h(s)\sup_{v\in B_K}|\partial_t f(v,s)|\dd s\\
&\leq C\int_{-\infty}^{+\infty} h(t)\|\partial_t f(v,t)\|_{W^{d+1,1}(B_K)}\dd t.\numberthis\label{eq:39}
\end{align*}
We set
\[f(v,t) = \left|\E_X\left[e^{itg_\lambda^X(v)}\right] - e^{-\frac{t^2}{2}}\right|^{2q}\quad\quad\text{and}\quad\quad h(t) = \frac{1}{\left(1+|t|^{2+\varepsilon}\right)^{2q}}.\]
in \eqref{eq:39}. By Fubini theorem,
\[\E_a\left[\sup_{v\in K}\sup_{t\in\R}h(t)f(v,t)\right]\leq C\sum_{|\alpha|\leq d+1}\int_{-\infty}^{+\infty}h(t)\int_{B_K}\E_a\left[|\partial_\alpha\partial_t f(v,t)|\right]\dd v\dd t.\numberthis\label{eq:9}\]
It remains to estimate the integrand. Using the derivative of the power function, we have
\[\partial_\alpha\partial_t f(v,t) = g(v,t)\left|\E_X\left[e^{itg_\lambda^X(v)}\right] - e^{-\frac{t^2}{2}}\right|^{2(q-d-2)},\]
for some function $g$ to be explicited. Using Cauchy--Schwarz inequality
\[\E_a\left[|\partial_\alpha\partial_t f(v,t)|\right] \leq \sqrt{\E_a[g(v,t)^2]}\;.\;\sqrt{\E_a\left[\left|\E_X\left[e^{itg_\lambda^X(v)}\right] - e^{-\frac{t^2}{2}}\right|^{4(q-d-2)}\right]}.\]
According to Theorem \ref{thm1}, there is a constant $C$ independent of $t$ and $\lambda$ such that
\begin{align*}
\sqrt{\E_a\left[\left|\E_X\left[e^{itg_\lambda^X(v)}\right] - e^{-\frac{t^2}{2}}\right|^{4(q-d-2)}\right]} &\leq C(1+|t|^{4(q-d-2)})\left(\frac{\eta(\lambda)}{\lambda}\right)^{q-d-2}.
\end{align*}
We will show that for some polynomial $P$ of degree $m$ independent of $q$,
\[\E_a[g(v,t)^2] \leq P(t).\numberthis \label{eq:10}\]
We will establish this fact in the end of the proof. Injecting this into the expression \eqref{eq:9}, we obtain
\[\E_a\left[\sup_{v\in K}\sup_{t\in\R}h(t)f(v,t)\right]\leq C\left(\frac{\eta(\lambda)}{\lambda}\right)^{q-d-2}\int_{-\infty}^{+\infty}\frac{(1+|t|^{\frac{m}{2}})(1+|t|^{4(q-d-2)})}{\left(1+|t|^{2+\varepsilon}\right)^{2q}}\dd t,\]
and for $q$ large enough, the integral is bounded. The end of the proof is the same as in the remark following Theorem \ref{thm1}. We have by Markov inequality and $q$ large enough,
\begin{align*}
\displaystyle{\PP_a\left(\sup_{t\in\R}\sup_{v\in K}\frac{\left|\E_X\left[e^{itg_\lambda^X(v)}\right] - e^{-\frac{t^2}{2}}\right|}{1+|t|^{2+\varepsilon}}>\lambda^{\varepsilon}\left(\frac{\eta(\lambda)}{\lambda}\right)^{1/2}\right)} &\leq \left(\frac{\lambda}{\eta(\lambda)\lambda^{2\varepsilon}}\right)^q\displaystyle{\E_a\left[\displaystyle{\sup_{v\in K}}\sup_{t\in\R}h(t)f(v,t)\right]} \\
\\
&\leq C \frac{1}{\lambda^{2q\varepsilon}}\left(\frac{\eta(\lambda)}{\lambda}\right)^{d+2}.\\
\end{align*}
For $q$ large enough the right-hand term is summable, and Borel--Cantelli lemma implies the existence a constant $C(\omega)$ depending only on $K$ and $\varepsilon$ such that
\[\left|\E_X\left[e^{itg_\lambda^X(v)}\right] - e^{-\frac{t^2}{2}}\right|\leq  C(\omega)(1+|t|^{2+\varepsilon})\frac{\sqrt{\eta(\lambda)}}{\lambda^{\frac{1}{2}-\varepsilon}}.\]
It remains to show the estimate \eqref{eq:10}. We have
\begin{align*}
\left|\E_X\left[e^{itg_\lambda^X(v)}\right] - e^{-\frac{t^2}{2}}\right|^{2q} = \E_{\textbf{X}}\left[\prod_{k=1}^{2q}\left(e^{\pm itg_\lambda^{X_k}(v)}-e^{-\frac{t^2}{2}}\right)\right].
\end{align*}
From this expression we deduce that
\[g(v,t) = \E_\textbf{X}\left[F\left(t,\left(\partial_\alpha g_\lambda^{X_k}\right)_{\substack{|\alpha|\leq d'\\1\leq k\leq 2q}}\right)\right],\]
with $F$ a function bounded by a polynomial of degree $2(d+2)$ in its arguments. But the partial derivatives of $g_\lambda$ are still Gaussian under $\PP_a$, and the local Weyl law (see also the expression \eqref{eq:11}) implies the existence of a universal constant $C$ such that
\[\E_a\left[\left(\partial_\alpha g_\lambda^{X_k}\right)^{2p}\right] = \frac{(2p)!}{2^pp!}\E_a\left[\left(\partial_\alpha g_\lambda^{X_k}\right)^2\right]^p \leq \frac{(2p)!}{2^pp!}C,\]
It implies that
\[\E_a\left[g(v,t)^2\right]\leq P(t),\]
for some polynomial $P$ in $t$ whose degree is independent of $q$.\jump

\end{proof}

\section{Proof of tightness estimates}\label{append.tight}

\begin{proof}[Proof of Lemma \ref{lemme3}]
Let $x,y\in\M$. We have
\begin{align*}
\E_a\left[\partial_\alpha g_\lambda^x(u)\,\partial_\alpha g_\lambda^y(v)\right] &= \frac{1}{K(\lambda)\lambda^{2|\alpha|}}\sum_{\lambda_n\leq \lambda}\left(\partial_\alpha\left(\varphi_n\circ\Phi_x\right)\left(\frac{u}{\lambda}\right)\right)\left(\partial_\alpha\left(\varphi_n\circ\Phi_y\right)\left(\frac{v}{\lambda}\right)\right).
\end{align*}
Setting
\[x_u = \Phi_x\left(\frac{u}{\lambda}\right)\quad\quad\text{and}\quad\quad y_v = \Phi_y\left(\frac{v}{\lambda}\right),\]
and using the fact that $\dd \exp_x = \Id$, we obtain{
\[
\E_a\left[\partial_\alpha g_\lambda^x(u)\,\partial_\alpha g_\lambda^y(v)\right] = \frac{1}{K(\lambda)\lambda^{2|\alpha|}}\left(\partial_{\alpha,\alpha} K_\lambda(x_u,y_v)\right) + O\left(\frac{1}{\lambda}\right).\]}
Recalling that the kernel $\mathcal{B}_d$ (resp. $\mathcal{S}_d$) is the $\CC^\infty$ scaling limit of the spectral projector we have
\[\E_a\left[\partial_\alpha g_\lambda^X(u)\,\partial_\alpha g_\lambda^Y(v)\right] = \partial_{\alpha,\alpha}\left[ \mathcal{B}_d(\lambda\,\dist(x_u,y_v))\right]+O\left(\frac{\eta(\lambda)}{\lambda}\right).\]
We briefly describe the $\CC^\infty$ extension of decorrelation estimates given by Lemma \ref{lemme1bis}. The proof is very similar and we refer to the proof of Lemma \ref{lemme1bis} for more details. Firstly, by the local Weyl law in the $\CC^\infty$ topology, we have uniformly on $x\in\M$ and $u\in B$,
\[\E_a\left[\left(\partial_\alpha g_\lambda^X(u)\right)^2\right] = C_\alpha + O\left(\frac{\eta(\lambda)}{\lambda}\right)\quad\text{and}\quad C_\alpha = \left.\partial_{\alpha,\alpha} \mathcal{B}_d(\|u-v\|)\right|_{u=v=0}.\numberthis \label{eq:11}\]
Secondly, take $X$ and $Y$ two independent uniform random variables on $\M$, and $k\geq 1$. As in Lemma \ref{lemme1bis}, we write
\begin{align*}
\E_X\left[\E_a[\partial_\alpha g_\lambda^X(u)\,\partial_\alpha g_\lambda^Y(v)]^{2k}\right] = I_1 + I_2,
\end{align*}
with
\begin{align*}
I_1=\int_{\dist(x,Y)>\varepsilon} \E_a[\partial_\alpha g_\lambda^X(u)\,\partial_\alpha g_\lambda^Y(v)]^{2k}\dd x \quad\text{and}\quad I_2 = \int_{\dist(x,Y)<\varepsilon}\E_a[\partial_\alpha g_\lambda^X(u)\,\partial_\alpha g_\lambda^Y(v)]^{2k}\dd x.
\end{align*}
Using a similar argument as in Lemma \ref{lemme1bis}, we deduce that uniformly on $u,v\in B$,
\[\E_X\left[\E_a[\partial_\alpha g_\lambda^X(u)\,\partial_\alpha g_\lambda^Y(v)]^{2k}\right] = O\left(\frac{\eta(\lambda)}{\lambda}\right).\numberthis\label{eq:30}\]
Define
\[W_X:u\mapsto \partial_\alpha g_\lambda^X(u)\quad\text{and}\quad W_Y:u\mapsto \partial_\alpha g_\lambda^Y(u).\]
The joint process $(W_X,W_Y)$ is Gaussian under $\PP_a$. We fix $u,v\in B$ and set
\[\rho(u,v) = \frac{\E_a[W_X(u)W_Y(v)]^2}{\E_a[W_X(u)^2]\E_a[W_Y(v)^2]}.\]
A direct Gaussian computation shows that
\[\E[W_X(u)^{2p}] = \frac{(2p)!}{2^pp!}\E\left[W_X(u)^2\right]^p\quad\text{and}\quad \E[W_Y(v)^{2p}] = \frac{(2p)!}{2^pp!}\E\left[W_Y(v)^2\right]^p,\]
and
\[\frac{\E[(W_X(u)W_Y(v))^{2p}]}{\E_a\left[W_X(u)^{2p}\right]\E_a\left[W_Y(v)^{2p}\right]} = \sum_{k=0}^p
\frac{\begin{pmatrix}
2p+2k \\ 
2p
\end{pmatrix}\begin{pmatrix}
2p \\ 
p+k
\end{pmatrix}}{\begin{pmatrix}
2p \\ 
p
\end{pmatrix} }
\rho(u,v)^k(1-\rho(u,v))^{p-k}:=Q_p(\rho(u,v)).\numberthis\label{eq:29}\]
From identity \eqref{eq:29} we compute
\[\E_a\left[\left(\E_X\left[\int_{B} W_X(u)^{2p}\dd u\right]-\frac{(2p)!}{2^pp!}(C_\alpha)^p\right)^2\right] = \left(\frac{(2p)!}{2^pp!}\right)^2(\Delta_1+\Delta_2),\numberthis\label{eq:31}\]
with
\[\Delta_1 := \left(\E_{X}\left[\int_{B}\E_a\left[W_X(u)^2\right]^p\dd u\right]-(C_\alpha)^p\right)^2,\]
and
\[\Delta_2 := \int_{B}\int_{B}\E_{X,Y}\left[\E_a\left[W_X(u)^2\right]^p\E_a\left[W_Y(v)^2\right]^p\left(Q_p(\rho(u,v))-1\right)\vphantom{\int}\right]\dd u\dd v.\]
From Equation \eqref{eq:11} we have
\[\Delta_1 = O\left(\left(\frac{\eta(\lambda)}{\lambda}\right)^2\right).\]
As for the term $\Delta_2$, we use the fact that $\E_a\left[W_X(u)^2\right]$ is bounded above and below by positive constants for $\lambda$ large enough, from equation \eqref{eq:11}. We develop the polynomial $Q_p$ and we use equation \eqref{eq:30} to obtain
\begin{align*}
\E_{X,Y}\left[\E_a\left[W_X(u)^2\right]^p\E_a\left[W_Y(v)^2\right]^p\left(Q_p(\rho(u,v))-1\right)\vphantom{\int}\right]&\leq C\,\E_{X,Y}\left[|Q_p(\rho(u,v))-1|\right]\\
&\leq C\E_{X,Y}\left[\sum_{k=1}^p |p_k| \rho(u,v)^k\right]\\
&= O\left(\frac{\eta(\lambda)}{\lambda}\right).
\end{align*}
Since the estimate is uniform on $u,v$ we deduce
\[\Delta_2 = O\left(\frac{\eta(\lambda)}{\lambda}\right).\]
Injecting this estimate into identity \eqref{eq:31} we obtain
\[\E_a\left[\left(\E_X\left[\int_{B} W_X(u)^{2p}\dd u\right]-\frac{(2p)!}{2^pp!}(C_\alpha)^p\right)^2\right] = \;O\left(\frac{\eta(\lambda)}{\lambda}\right),\]
The quantity inside the square is a polynomial of degree at most $2p$ in the Gaussian random variables $(a_n)_{n\geq0}$, and hence belongs to a finite fixed sum of Wiener chaos. The hypercontractivity property asserts that for such a polynomial, all the $L^q$ norms for $q\geq2$ are equivalents, which in our case implies that for every $q\geq 2$,
\[\E_a\left[\left(\E_X\left[\int_{B} W_X(u)^{2p}\dd u\right]-\frac{(2p)!}{2^pp!}(C_\alpha)^p\right)^q\right]   = O\left(\left(\frac{\eta(\lambda)}{\lambda}\right)^{q/2}\right).\]
For more details on Wiener chaos and hypercontractivity we refer the reader to the book \cite{No12}. Borel--Cantelli lemma implies the existence for every $\varepsilon>0$ of a constant $C(\omega)$ independent of $\lambda$ such that
\[\left|\E_X\left[\int_{B} W_X(u)^{2p}\dd u\right]-\frac{(2p)!}{2^pp!}(C_\alpha)^p\right| \leq C(\omega)\left(\frac{\eta(\lambda)}{\lambda}\right)^{1/2-\varepsilon},\]
which in turn implies the existence of a constant $\widetilde{C}(\omega)$ such that
\[\sup_{\lambda>0}\E_X\left[\int_{B} |\partial_\alpha g_\lambda^X(u)|^{2p}\dd u\right]\leq \widetilde{C}(\omega).\]
\end{proof}
\section{Proof of Wassertein estimates}\label{append.wass}

\begin{proof}[Proof of Lemma \ref{lemme8}]
Let $\phi$ be a function in $\mathcal{S}(\R)$, supported on the compact $[-(M+1),M+1]$. Using Plancherel isometry we have (the constant $C$ may change from line to line)
\begin{align*}
\left|\E[\phi(X_n)]-\E[\phi(X)]\right| &\leq \frac{1}{2\pi}\int_\R\left|\E\left[e^{itX_n}\right]-\E\left[e^{itX}]\right]\right||\widehat{\phi}(t)|\dd t\\
&\leq \frac{C}{n^\alpha}\int_\R (1+|t|^m)|\widehat{\phi}(t)|\dd t\\
&\leq \frac{C}{n^\alpha}\int_\R (1+|t|^{m+1})|\widehat{\phi}(t)|\frac{1}{1+|t|}\dd t\\
&\leq \frac{C}{n^\alpha}\int_\R |\widehat{\phi}(t)|\frac{\dd t}{1+|t|}\;\; +\;\; \frac{C}{n^{\alpha}}\int_\R |t|^{m+1}|\widehat{\phi}(t)|\frac{\dd t}{1+|t|}\\
&\leq \frac{C}{n^\alpha}\sqrt{\int_\R |\widehat{\phi}(t)|^2\dd t}\;\;+\;\;\frac{C}{n^\alpha}\sqrt{\int_\R |t|^{2m+2}|\widehat{\phi}(t)|^2\dd t}\quad\text{(Jensen)}\\
&\leq \frac{C}{n^\alpha}\|\phi\|_2 + \frac{C}{n^\alpha}\|\phi^{(m+1)}\|_2\quad\quad\quad\text{(Plancherel)}\\
&\leq C\frac{\sqrt{M+1}}{n^\alpha}\left(\|\phi\|_\infty + \|\phi^{(m+1)}\|_\infty\right).\numberthis\label{eq:34}
\end{align*}
By standard approximation argument, the inequality is true for every $\phi\in\CC^{m+1}(\R)$ with support in $[-(M+1),M+1]$. Suppose now that $\phi$ does not have compact support. Let $\chi_M$ a smooth function with support in $[-(M+1),M+1]$ such that $\chi_M = 1$ on $[-M,M]$. Set $\phi_M = \phi.\chi_M$. We write
\[\left|\E[\phi(X_n)]-\E[\phi(X)]\right| \leq \left|\E[\phi_M(X_n)]-\E[\phi_M(X)]\right| + \|\phi_\infty\|\PP(X_n>M) + \|\phi_\infty\|\PP(X>M).\]
From inequality \eqref{eq:34} and Markov inequality applied to the function $x\mapsto |x|^{\beta}$,
\[\left|\E[\phi(X_n)]-\E[\phi(X)]\right| \leq C\frac{\sqrt{M+1}}{n^\alpha}\left(\|\phi_M\|_\infty + \|\phi_M^{(m+1)}\|_\infty\right) + C\frac{\|\phi_\infty\|}{M^\beta}.\]
Using Leibniz rule, we have
\[\|\phi_M\|_\infty \leq \|\phi\|_\infty\quad\quad\text{and}\quad\quad\|\phi_M^{(m+1)}\|_\infty \leq C_m\sup_{k\leq m+1}\|\phi^{(k)}\|_\infty.\]
Choosing
\[M = n^{\frac{\alpha}{\beta+\frac{1}{2}}}\]
and under the requirement that $M>1$, we obtain
\[\left|\E[\phi(X_n)]-\E[\phi(X)]\right|\leq Cn^{-\frac{2\alpha\beta}{2\beta+1}}\sup_{k\leq m+1}\|\phi^{(k)}\|_\infty,\]
from which it follows that
\[\Wass_{(m+1)}(X_n,X) \leq Cn^{-\frac{2\alpha\beta}{2\beta+1}}.\]
\end{proof}
\section{Proof of the geometric lemma}\label{append.geom}
\begin{proof}[Proof of Lemma \ref{lemme12}]
Both sides are dimensionless and it suffices to prove the assertion for $a=1$. We can assume that $\HH^{d-1}(E\cap \partial D)=0$, else we could find a segment $S$ passing through one of the vertices and such that
$\HH^1(E\cap S\cap \partial D)>0$, and in that case the result is true.\jump

We will prove Lemma \ref{lemme12}  by a probabilistic method. We denote $(A_j)_{1\leq j\leq 2^d}$ the vertices of the cube. Let $P$ be a point chosen uniformly randomly on the cube $[0,1]^d$. Let $(A_jP)$ be the random line passing through the points $A_j$ and $P$. We will in fact prove that
\[\E\left[\sum_{j=1}^{2^d} \Card\left\lbrace E\cap (A_jP)\cap D\right\rbrace\right] \geq \frac{1}{2d}\HH^{d-1}(E\cap D), \numberthis \label{eq:19}\]
which implies the result, since for some realization of $P$ and some $j$ we must have
\[\Card(E\cap (A_jP)\cap D) \geq \frac{1}{2^{d+1}d}\HH^{d-1}(E\cap D).\]
Since we assumed that $\HH^{d-1}(E\cap \partial D)=0$ we can suppose that $E\subset \mathring{D}$. Since the manifold $E$ has bounded curvature, it is a doubling space and Vitali--Lebesgue covering theorem (see \cite[p.~4]{He01}) asserts that for all $r_0>0$, we can find a disjoint family of (relatively compact) geodesic balls $(E_{r_n})_{n\geq 0}$ in $E$ such that the geodesic ball $E_{r_n}$ has radius $r_n<r_0$, and such that
\[
\HH^{d-1}\left(E \setminus\left(\bigsqcup_{n\in \N} E_{r_n} \right)\right) = 0.
\]
By linearity of both sides of \eqref{eq:19} and monotone convergence, it is sufficient to prove the inequality \eqref{eq:19} by replacing $E$ with $E_r$, a small (relatively compact) geodesic ball of radius $r<r_0$ centered at some point $x\in E$.  For $r$ sufficiently small, the geodesic ball $E_r$ is comparable to a $\R^{d-1}$-ball. More precisely, set $B_r(x) = \exp^{-1}_x(E_r)$. Riemannian volume comparison theorems asserts that
\[\HH^{d-1}(B_r(x)) = \HH^{d-1}(E_r)(1+o(r_0)),\]
and the estimate is uniform on $E$ by the curvature bound assumption. We will prove that for some $j\in \lbrace1,\ldots,2^d\rbrace$,
\[\E\left[\sum_{j=1}^{2^d} \Card\left\lbrace E_r\cap (A_jP)\right\rbrace\right] \geq \E\left[\Card\left\lbrace E_r\cap (A_jP)\right\rbrace\right] \geq \frac{1}{2d}\HH^{d-1}(E_r).\]
Let $\textbf{n}_x$ denote a normal unit vector at $x$. A little geometry shows that we can choose $j$ such that
\[|\langle \overrightarrow{A_jx},\textbf{n}_x\rangle|\geq\frac{1}{2}.\numberthis \label{eq:20}\]
For $r_0$ small enough and uniformly on $E$, the hypersurface $E_r$ is almost flat and the line $(A_j P)$ has at most one point of intersection with $E_r$. The opposite would imply that for some $y\in E_r$ the line $(A_jy)$ is tangent to $E_r$ at some point $y$. That is $\langle \overrightarrow{A_jy},\textbf{n}_y\rangle = 0$. But it contradicts the inequality \eqref{eq:20} and the continuity on $E_r$ of the mapping
\[x\mapsto |\langle \overrightarrow{A_jx},\textbf{n}_x\rangle|.\]
The uniformity of $E$ comes from the fact that the modulus of continuity of this application is controlled by the curvature of $E$. We deduce
\[\E\left[\Card\left\lbrace E_r\cap (A_jP)\right\rbrace\right] = \PP\left(\Card\left\lbrace E_r \cap (A_jP)\right\rbrace \neq \void\right).\]
Uniformly on $x$ in $E$, we can find $r' = r+o(r_0)$ such that every line that passes through $A_j$ and intersects the $d-1$-dimensional ball $B_{r'}(x)\subset T_x E$, also passes through $E_r$. Indeed, the central projection of $E_r$ onto $T_x E$ with center of projection $A_j$ is almost a $\R^{d-1}$-ball and must contain a ball of radius $r' = r + o(r_0)$ (see Figure \ref{fig2}).\jump

\begin{figure}[ht]
\begin{center}
\begin{tikzpicture}[xscale=0.8, yscale=0.8]
    \draw  (0,0) -- (8,0);
    \draw  (0,8) -- (8,8);
    \draw  (0,8) -- (0,0);
    \draw  (8,0) -- (8,8);
    \draw (8,8) node[right]{$A_j$};
    \draw [line width=0.2mm,domain=40:90,samples=200] plot ({1+3*cos(\x)},{3*sin(\x)});
    \fill [black, domain=0:360] plot ({1+3*cos(65)+0.05*cos(\x)},{3*sin(65)+0.05*sin(\x)});
    \draw ({1+3*cos(65)},{3*sin(65)}) node[below]{$x$};
    \draw ({1+3*cos(65)+3*sin(65)},{3*sin(65)-3*cos(65)}) -- ({1+3*cos(65)-3*sin(65)},{3*sin(65)+3*cos(65)});
    \draw [dashed] (8,8) -- ({1+3*cos(40) + 0.32*(1+3*cos(40)-8)},{3*sin(40)+0.32*(3*sin(40)-8)});
    \draw [dashed] (8,8) -- ({1+3*cos(90)+ 0.15*(1+3*cos(90)-8)},{3*sin(90)+0.15*(3*sin(90)-8)});
    \draw [line width = 0.6mm] ({1+3*cos(65)+1.3*sin(65)},{3*sin(65)-1.3*cos(65)}) -- ({1+3*cos(65)-1*sin(65)},{3*sin(65)+1*cos(65)});
    \draw (4.6,1.9) node[right]{$T_xE$};
    \draw (2,3) node[right]{$B_{r'}(x)$};
    \draw (2.6,1.9) node[right]{$E_r$};
\end{tikzpicture}
\end{center}
\caption{Construction of $B_{r'}(x)$.} \label{fig2}
\end{figure}

\noindent We deduce
\[\PP\left(\Card\left\lbrace E_r \cap (A_jP)\right\rbrace \neq \void\right)\geq \PP\left(\Card\left\lbrace B_{r'}(x) \cap (A_jP)\right\rbrace \neq \void\right).\]
But the right hand side is easy to estimate. It is the volume of the cone in $[0,1]^d$, based at $A_j$ and generated by the ball $B_{r'}(x)$. The formula $\mathrm{base}\times\mathrm{height}/d$ gives
\begin{align*}
\PP\left(\Card\left\lbrace B_{r'}(x) \cap (A_jP)\right\rbrace \neq \void\right) &\geq \frac{\HH^{d-1}(B_{r'}(x))}{d}|\langle\overrightarrow{A_jx},\textbf{n}_x\rangle|\\
&\geq \frac{1}{2d}\HH^{d-1}(B_r(x))(1+o(r_0))\\
&\geq \frac{1}{2d}\HH^{d-1}(E_r)(1+o(r_0)).
\end{align*}
Patching up the above estimates we recover
\[\E\left[\sum_{j=1}^{2^d} \Card\left\lbrace E\cap (A_jP)\cap D\right\rbrace\right] \geq \frac{1}{2d}\HH^{d-1}(E\cap D)(1+o(r)),\]
and letting $r$ go to zero we deduce the result.
\end{proof}
\emergencystretch=1em
\printbibliography
\end{document}